\documentclass[12pt,a4paper]{amsart}
\usepackage{amsmath,ifthen,amsthm,srcltx,amsopn,amssymb,amsfonts}
\usepackage{amscd}

\usepackage[T1]{fontenc}
\usepackage[utf8x]{inputenc}
\usepackage{psfrag}
\usepackage{delarray,graphicx}
\usepackage{comment}
\usepackage[small,nohug]{diagrams}
\usepackage[all,cmtip]{xy}

\usepackage{epsfig,graphics}
\usepackage{graphicx}

\usepackage{amssymb,latexsym}
\usepackage{color}

\usepackage{amsthm}
\usepackage{eucal}

\newcommand{\showcomments}{yes}
\renewcommand{\showcomments}{no}

\newsavebox{\commentbox}
{\ifthenelse{\equal{\showcomments}{yes}}%
{\footnotemark
        \begin{lrbox}{\commentbox}
        \begin{minipage}[t]{1.25in}\raggedright\sffamily\tiny
        \footnotemark[\arabic{footnote}]}
{\begin{lrbox}{\commentbox}}}%
{\ifthenelse{\equal{\showcomments}{yes}}%
{\end{minipage}\end{lrbox}\marginpar{\usebox{\commentbox}}}
{\end{lrbox}}}

\title[Tetrahedra of flags, volume and homology of $\SL(3)$]{Tetrahedra of flags, volume \\ and homology of $\SL(3)$}

\author{Nicolas Bergeron, Elisha Falbel, and Antonin Guilloux} 
\thanks{N.B. is a member of the Institut Universitaire de France.}
\address{Institut de Math\'ematiques de Jussieu \\
Unit\'e Mixte de Recherche 7586 du CNRS \\
Universit\'e Pierre et Marie Curie \\
4, place Jussieu 75252 Paris Cedex 05, France \\}
\email{bergeron@math.jussieu.fr \\ falbel@math.jussieu.fr \\ aguillou@math.jussieu.fr}
\urladdr{http://people.math.jussieu.fr/~bergeron \\ http://people.math.jussieu.fr/~falbel \\ http://people.math.jussieu.fr/~aguillou}

 \DeclareFontFamily{OT1}{rsfs}{}
 
\DeclareFontShape{OT1}{rsfs}{n}{it}{<-> rsfs10}{}
\DeclareMathAlphabet{\mathscr}{OT1}{rsfs}{n}{it}

\newcommand{\Z}{\mathbb{Z}}

\newcommand{\m}{\mathfrak{m}}

\DeclareFontFamily{OT1}{rsfs}{}

\DeclareFontShape{OT1}{rsfs}{n}{it}{<-> rsfs10}{}
\DeclareMathAlphabet{\mathscr}{OT1}{rsfs}{n}{it}
\newcommand{\Q}{\mathbb{Q}}

\newcommand{\R}{\mathbb{R}}

\swapnumbers
\newtheorem{theorem}[subsection]{Theorem}
\newtheorem{lemma}[subsection]{Lemma}
\newtheorem{proposition}[subsection]{Proposition}

\newtheorem{cor}[subsection]{Corollary}

\newtheorem*{theorem*}{Theorem}
\newtheorem*{appl*}{Application}

\theoremstyle{definition}

\theoremstyle{remark}
\newtheorem*{remark}{Remark}

\newcommand\bZ{\mathbb{Z}}
\newcommand\bC{\mathbb{C}}
\newcommand\bR{\mathbb{R}}
\newcommand\bS{\mathbb{S}}
\newcommand\SL{\textrm{SL}}
\newcommand\PGL{\textrm{PGL}}
\newcommand\PU{\textrm{PU}}
\newcommand\Fl{\mathcal{F}l}
\newcommand\AFl{\mathcal{AF}l}
\renewcommand\P{\mathbb{P}}
\newcommand\C{\mathcal{C}}
\newcommand\pB{\mathcal{P}}
\newcommand\B{\mathcal{B}}
\newcommand\Vol{\textrm{Vol}}

\numberwithin{equation}{subsection}

\newcommand{\A}{\mathbb A}

\def\adots{\mathinner{\mkern2mu\raise1pt\hbox{.}
\mkern3mu\raise4pt\hbox{.}\mkern1mu\raise7pt\hbox{.}}}

\begin{document}

\begin{abstract}  
In the paper we define a ``volume'' for simplicial complexes of flag tetrahedra. This generalizes and unifies the classical volume of hyperbolic manifolds and the volume of CR tetrahedra complexes considered in \cite{falbel,FalbelWang}. We describe when this volume belongs to the Bloch group and more generally describe a variation formula in terms of boundary data. In doing so, we recover and generalize results of Neumann-Zagier \cite{NeumanZagier}, Neumann \cite{Neuman}, and
Kabaya \cite{Kabaya}. Our approach is very related to the work of Fock and Goncharov \cite{FG,FG2}.
\end{abstract}
\maketitle
\tableofcontents

\section{Introduction}

It follows from Mostow's rigidity theorem that the volume of a complete hyperbolic 
manifold is a topological invariant.  In fact, 
it coincides with Gromov's purely topological definition of simplicial volume. 
If the complete hyperbolic manifold $M$ has cusps,
Thurston showed that one could obtain complete hyperbolic structures on manifolds obtained from $M$ by Dehn surgery 
by gluing a solid torus with a sufficiently long geodesic.  
Thurston's framed his results for more general deformations
which are not complete hyperbolic manifolds, the volume of the deformation being the volume of its metric completion. 
 Neumann and Zagier \cite{NeumanZagier} and afterwards Neumann \cite{Neuman} provided a deeper analysis of these deformations and their volume.  
In particular, they showed that the variation of the volume
depends only on the geometry of the boundary and they  gave a precise formula 
for that variation in terms of the boundary holonomy.
  
It is natural to consider an invariant associated to a hyperbolic structure defined on 
the pre-Bloch group $\mathcal {P}(k)$ which is defined
as the abelian group generated by all the points in $k\setminus\{0,1\}$ where $k$ 
is a field 
quotiented by the 5-term relations (see section \ref{bloch} for definitions and references).   The volume 
function is well defined as a map $\mathrm{Vol} : \mathcal {P}(k)\rightarrow \R$ using the dilogarithm. 
The Bloch group $\mathcal {B}(k)$
is a subgroup of the pre-Bloch group $\mathcal {P}(k)$.  It is defined as the kernel of the
map
$$
\delta : \mathcal {P}(k)\rightarrow k^{\times}\wedge_\Z k^{\times}
$$
given by $\delta([z])=z\wedge (1-z)$. 
The
volume and the Chern-Simons invariant 
can then be seen through a function (the Bloch regulator)
$$
 \mathcal {B}(k)\rightarrow {\bC}/{\Q}.
$$
The imaginary part being related to the volume and the real part related to Chern-Simons 
$CS$ mod $\Q$ invariant.

Several extensions of Neumann-Zagier results were obtained.  Kabaya \cite{Kabaya} defined
an invariant in $\mathcal {P}(\bC)$ associated to a hyperbolic 3-manifold with 
boundary and obtained a 
description of the variation of the volume function which depends only on the boundary 
data.  Using different coordinates and methods Bonahon \cite{Bonahon2}  showed a similar formula.

The volume function was extended in \cite{falbel,FalbelWang} in order to deal with Cauchy-Riemann (CR) structures.
  More precisely,
consider $\bS^3\subset \bC^2$ with the contact structure obtained as the intersection $D=TS^3\cap JTS^3$ where $J$ is the
multiplication by $i$ in $\bC^2$.
The operator $J$ restricted to $D$ defines the standard CR structure on $\bS^3$.
The group of CR-automorphisms of $\bS^3$ is $\PU(2,1)$ and we say
that a manifold $M$ has a spherical CR structure if it has a $(\bS^3, \PU(2,1))$-
geometric structure.  Associated to a CR triangulation it was defined in \cite{FalbelWang} an
 invariant in $\mathcal {P}(\bC)$ which is in the Bloch group in case
 the structure has unipotent boundary holonomy.  The definition of that invariant is valid for ``cross-ratio structures'' (which includes hyperbolic and CR structures) as defined in \cite{falbel}.  It turns out to be a coordinate description of the decorated triangulations described bellow and the invariant in  $\mathcal {P}(k)$ coincides with the one defined before up to a multiple of four.

We consider in this paper a geometric framework
 which includes both hyperbolic structures
and CR structures on manifolds. We are in fact dealing with representations of the fundamental group of the manifold in $\PGL(3, \bC)$ that are parabolic : the peripheral holonomy should preserve a flag in $\bC^3$. Recall that a flag in $\bC^3$ is a line in a plane of $\bC^3$. The consideration of these representations links us to the work of Fock and Goncharov \cite{FG,FG2}. Indeed we make an intensive use of their combinatorics on the space of representations of surface groups in $\SL(3 , \R)$.

As in the original work of Thurston and Neumann-Zagier, we work with decorated triangulations. 
Namely, let $T$ be a triangulation of a 3-manifold $M$.
To each tetrahedron we associate a quadruple of flags 
(corresponding to the four vertices) in $\bC^3$. In the case 
of ideal triangulations, where the
manifold $M$ is obtained from the triangulation by deleting the vertices, 
we impose that the holonomy
around each vertex preserves the flag decorating this vertex. Such a decorated triangulation gives a set of flag coordinates, and more precisely two sets: affine flag coordinates $a$ and projective flag coordinates $z$. Those are, in the Fock and Goncharov setting, the $a$- and $z$-coordinates on the boundary of each tetrahedron, namely a  four-holed sphere. 

The main result in this paper is the construction of an element  $\beta \in \mathcal {P}(\bC)$ associated to a decorated triangulation and a description of a precise formula for $\delta (\beta)$ in terms of boundary data. This formula is given in theorem \ref{thm:kabaya}.

The core of the proof of theorem \ref{thm:kabaya} goes along the same lines of the homological proof of Neumann \cite{Neuman}. We nevertheless believe that the use  of the combinatorics of Fock and Goncharov  sheds some light on Neumann's work. The two theories fit well together, allowing a new understanding, in particular, of the ``Neumann-Zagier'' symplectic form. 

The organisation of the paper is as follows. In  section \ref{configurations} we describe flags and configurations of flags.
Following \cite{FG}, we define $a$- and $z$-coordinates for configurations of flags.  These data 
define a decorated tetrahedron. In section \ref{bloch} we define 
an element in the pre-Bloch group associated to a decorated 
tetrahedron (cf. also \cite{FalbelWang}). We then define the  
volume of a decorated tetrahedron and show how 
previous definitions in hyperbolic and CR geometry are included in this context.  
We moreover relate our work to Suslin's work on $K_3$, showing that our volume map is essentially Suslin map from $H_3(\SL(3))$ to the Bloch group. This gives a geometric and intuitive construction of the latter. Here we are very close to the work of Zickert on the extended Bloch group \cite{Zickert}.
In the next section \ref{decoration} we associate to a decorated tetrahedron $T$ the element 
$\delta (\beta(T))$ and compute it using both $a$-coodinates and $z$-coordinates.

This local work being done, we move on in section \ref{decoration2} to the framework of decorated simplicial complexes.
The decoration consists of $a$-coordinates or $z$-coordinates associated to each tetrahedron 
and satisfying appropriate compatibility conditions along edges and faces. 
The main result is the computation of
 $W=\delta(\beta(M))$  which turns out to depend only on boundary data (Theorem \ref{thm:kabaya}).
 
We first give a proof of Theorem \ref{thm:kabaya} when the decoration is unipotent. We then deal with the proof of the general case. In doing
so we have to develop a generalization of the Neumann-Zagier bilinear relations to the $\PGL (3, \bC)$ case. In doing so the Goldman-Weil-Petersson
form for tori naturally arises. We hope that our proof sheds some light on the classical $\PGL (2, \bC)$ case. 

In section \ref{examples} we describe all
unipotent decorations on the complement of 
the figure eight knot. It was proven by P.-V. Koseleff that there are a finite number of 
unipotent structures and all of them are either hyperbolic  
or CR. The natural question of the rigidity of unitotent representation will be investigated in a forthcoming paper \cite{BFG-rigidity} (see also \cite{Genzmer}).

Finally in section \ref{applications}, we describe applications of theorem \ref{thm:kabaya}. First, we follow again Neumann-Zagier and obtain an explicit formula for the variation of the 
volume function which depends on boundary data. Then, relying on remarks of Fock and Goncharov, we describe a $2$-form on the space of representations of the boundary of our variety which coincides with Weil-Petersson form in some cases (namely for hyperbolic structures and unipotent decorations).

We thank J. Genzmer, P.-V. Koseleff and Q. Wang for fruitful discussions.

\section{Configurations of flags and cross-ratios}\label{configurations}

We consider in this section the flag variety $\Fl$ and the affine flag variety $\AFl$ of $\SL(3)$ over a field $k$. We define coordinates on the configurations of $4$ flags (or affine flags), very similar to the coordinates used by Fock and Goncharov \cite{FG}.

\subsection{Flags, affine flags and their spaces of configuration}

We set up here notations for our objects of interest. 
Let $k$ be a field and $V=k^3$. A flag in $V$ is usually seen as a line and a plane, the line belonging to the plane. We give, for commodity reasons, the following alternative description using the dual vector space $V^*$ and the projective spaces $\P(V)$ and $\P(V^*)$:

We define the spaces of {\it affine flags} $\AFl(k)$ and {\it flags} $\Fl(k)$ by the following:
\begin{eqnarray*}
 \AFl(k) &= & \{(x,f)\in (V\setminus\{0\})\times (V^*\setminus\{0\}) \textrm{ such that } f(x)=0\} \\
 \Fl(k) & = & \{([x],[f])\in \P(V)\times \P(V^*) \textrm{ such that } f(x)=0\}.
\end{eqnarray*}

The space of flags $\Fl(k)$ is identified with the homogeneous space $\PGL(3,k)/ B$, where $B$ is the Borel subgroup of upper-triangular matrices in $\PGL(3,k)$. Similarly, the space of affine flags $\AFl(k)$ is identified with the homogeneous space $\SL(3,k)/ U$, where $U$ is the subgroup of  unipotent upper-triangular matrices in $\SL(3,k)$.

\subsection{} 
Given a $G$-space $X$, we classicaly define the configuration module of ordered points in $X$ as follows.
For $n\geq 0$, let
$C_n(X)$ be the free abelian group
generated by the set 
$$
(p_0,\cdots,p_{n})\in X^{n+1}
$$
of all ordered $(n+1)$ set of points in
$X$.   The group $G$ acts on $X$ and therefore
 also acts diagonally on $C_n(X)$ giving it a
left $G$-module structure.

We define the differential $d_n: C_n(X)\rightarrow C_{n-1}(X)$
by
$$
 d_n(p_0,\dots,p_n)=\sum_{i=0}^{n}(-1)^i (p_0,\dots,\hat{p_i},
 \dots, p_n),
$$
 then we can check that every $d_n$ is a $G$-module homomorphism and
$d_n\circ d_{n+1}=0$.
 Hence we have the $G$-complex
 $$
    C_{\bullet}(X): \cdots \rightarrow C_n(X)\rightarrow C_{n-1}(X)\rightarrow
    \cdots \rightarrow C_0(X).
 $$
The augmentation map $\epsilon: C_0(X)\rightarrow \bZ $ is defined on generators by 
 $\epsilon(p)=1$ for each $p\in X$.
If $X$ is infinite, the augmentation complex is exact.

For a left $G$-module $M$, we denote $M_G$ its group of co-invariants, that
is,
\[
 M_{G}=M/\langle gm-m, \forall g\in G, m\in M \rangle.
\]
Taking the co-invariants of the complex
$C_{\bullet}(X)$, we get the induced complex:
$$
    C_{\bullet}(X)_G: \cdots \rightarrow C_n(X)_G\rightarrow C_{n-1}(X)_G\rightarrow
    \cdots \rightarrow C_0(X)_G,
$$
with differential $\bar{d}_n: C_n(X)_G\rightarrow
C_{n-1}(X)_G$ induced by $d_n$. 
We call $H_\bullet(X)$ the homology of this complex.


\subsection{} We let now $G=\PGL(3,k)$ and $X=\Fl$.
For every integer $n\geq 0$, the $\bZ$-module of {\it coinvariant 
configurations of $n+1$ ordered  flags} is defined by:
$$\C_{\bullet}(\Fl) = C_{\bullet}(\Fl)_G.$$
The natural projection $\pi \: : \: \SL(3)\rightarrow \PGL(3) \to \PGL(3)/B=\Fl$ 
gives a map 
$$
\pi_* \: : \: H_3(\SL(3)) \to H_3(\Fl).
$$

We will study in this paper the homology groups $H_3(\SL(3,k))$ 
(which is the third group of discrete homology of $\SL(3,k)$), 
$H_3(\AFl)$ and $H_3(\Fl)$.

It is usefull to consider a subcomplex of $\C_{\bullet}(\Fl)$ of generic configurations which 
contains all the information about its homology.  We leave to the reader the 
verification that indeed the definition below gives rise to subcomplexes of $C_3(\Fl)$
and $\C_3(\Fl)$   

\subsection{} A {\it generic configuration of  flags} $([x_i],[f_i])$, $1\leq i\leq n+1$ 
is given by $n+1$ points $[x_i]$ in general position and $n+1$ lines $\textrm{Ker }f_i$
in $\P(V)$ such that $f_j(x_i)\neq 0$ if $i\neq j$.
We will denote $C^r_n(\Fl)\subset C_n(\Fl)$ and  $\C^r_n(\Fl)\subset \C_n(\Fl)$ the corresponding
module of configurations and its coinvariant module by the diagonal action by $\SL(3)$.

A configuration of ordered points in $\P(V)$ is said to be in {\it general position} when they are all distinct and no three points are contained in the same line.  Observe that the genericity condition of flags does not imply that the lines are in a general position.

\subsection{} \label{normal}
Since $G$ acts
transitively on $C^r_1\Fl$, we see that $C^r_n(\Fl)_G=\bZ$ if $n\leq 1$, and
the differential $\bar{d}_1: C^r_1(\Fl)_G\rightarrow
C^r_{0}(\Fl)_G$ is zero.

In order to describe $\C^r_2(\Fl)$ consider a configuration of $3$ generic flags $([x_i ] , [f_i ] )_{1 \leq i \leq 3} \in \C^r_2(\Fl)$.
One can then define a projective coordinate system of $\P (\bC^3)$: take the one where the point $x_1$ has coordinates $[1:0 : 0]^t$, 
the point $x_2$ has coordinates $[0:0:1]^t$, the point $x_3$ has coordinates $[1:-1:1]^t$ and the intersection of $\mathrm {Ker} f_1$ and $\mathrm {Ker} f_2$ has coordinates $[0 : 1 : 0]^t$. The line $\mathrm {Ker} f_3$ then 
has coordinates $[z :  z+1 : 1]$ where
$$z=  \frac{f_1(x_2)f_2(x_3)f_3(x_1)}{f_1(x_3)f_2(x_1)f_3(x_2)} \in k^{\times}$$
is the {\it triple ratio}. We have $\C^r_2(\Fl)=\bZ[k^{\times}]$. Moreover
the differential  $\bar{d}_2 : C^r_2(\Fl)_G\rightarrow
C^r_{1}(\Fl)_G$ is given on generators $z\in k^{\times}$  by $\bar{d}_2(z)= 1$ and therefore $H_1(\Fl)=0$.

We denote by $z_{123}$ the triple ratio of a cyclically oriented triple of flags $([x_i] , [f_i] )_{i=1,2,3}$. Note that $z_{213} = 1/z_{123}$.  Observe that when $z_{123}=-1$
the three lines are not in general position.

\subsection{Coordinates for a tetrahedron of flags}\label{ss:coord-para}

We call a generic configuration of $4$ flags a \emph{tetrahedron of flags}. The coordinates we use for a tetrahedron of flags are the same as those used by Fock and Goncharov \cite{FG} to describe a flip in a triangulation. We may see it as a blow-up of the flip into a tetrahedron.  They also coincide with coordinates used in \cite{falbel} to describe a {\sl cross-ratio structure} on a tetrahedron (see also section \ref{CR}).

Let $([x_i],[f_i])_{1\leq i \leq 4}$ be an element of $\C_3(\Fl)$. Let us dispose symbolically these flags on a tetrahedron $1234$ (see figure \ref{tetra}). 
We define a set of 12 coordinates on the edges of the thetrahedron ($1$ for each oriented edge) and four coordinates associated to the faces.

\begin{figure}[ht]      
\begin{center}
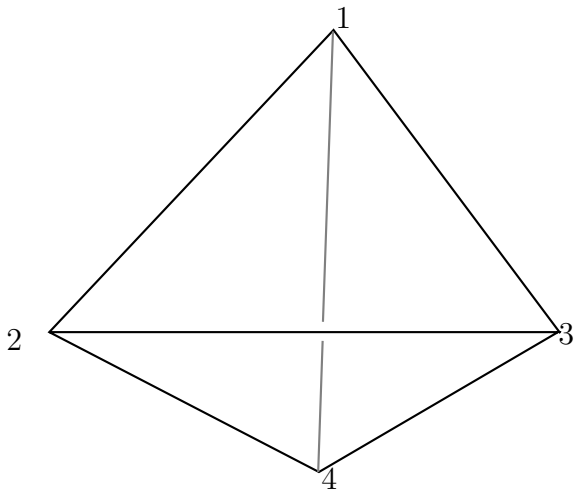
\caption{An ordered tetrahedron} \label{tetra}
\end{center}
\end{figure}

To define the coordinate $z_{ij}$ associated to the edge $ij$, we first define $k$ and $l$ such that the permutation $(1,2,3,4) \mapsto (i,j,k,l)$ is even. The pencil of (projective) lines through the point $x_i$ is a projective line $\P_1(k)$. We naturally have four points in this projective line: the line $\textrm{ker}(f_i)$ and the three lines through $x_i$ and one of the $x_l$ for $l\neq i$. We define $z_{ij}$ as the cross-ratio\footnote{Note that we follow the usual convention (different from the one used by Fock and Goncharov) that the cross-ratio of four points 
$x_1, x_2 , x_3 , x_4$ on a line is the value at $x_4$ of a projective coordinate taking
value $\infty$ at $x_1$, $0$ at $x_2$, and $1$ at $x_3$. So we employ the formula 
$\frac{(x_1 -x_3)(x_2-x_4)}{(x_1 -x_4)(x_2-x_3)}$ for the cross-ratio.} of these four points:    
$$z_{ij} := [\textrm{ker}(f_i),(x_ix_j),(x_ix_k),(x_ix_l)].$$

We may rewrite this cross-ratio thanks to the following useful lemma.

\begin{lemma} \label{fact1}
 We have $z_{ij}=\frac{f_i(x_k)\det(x_i,x_j,x_l)}{f_i(x_l)\det(x_i,x_j,x_k)}$.  Here 
 the determinant is w.r.t. the canonical basis on $V$.
\end{lemma}

\begin{proof}
 Consider the following figure: 
\begin{figure}[ht]      
\begin{center}
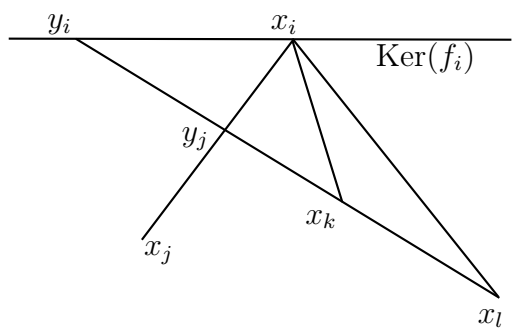
\caption{Cross-ratio} \label{formula}
\end{center}
\end{figure}

By duality, $z_{ij}$ is the cross-ratio between the points $y_i,y_j$ and $x_k,x_l$ on the line $(x_kx_l)$. Now, $f_i$ is a linear form vanishing at $y_i$ and $\det(x_i,x_j, \cdot)$ is a linear form vanishing at $y_ j$. Hence, on the line $(x_k x_l)$, the linear form $f_i (x)$ is proportional to $(\cdot -y_i)$ and 
$\det (x_i , x_j , \cdot )$ is proportional to $(\cdot - y_j )$. This proves the formula.
\end{proof}

\subsection{} Each face $(ijk)$ inherits a canonical orientation as the boundary of the tetrahedron $(1234)$. Hence to the face $(ijk)$, we associate the $3$-ratio
of the corresponding cyclically oriented triple of flags:
$$z_{ijk} = \frac{f_i(x_j)f_j(x_k)f_k(x_i)}{f_i(x_k)f_j(x_i)f_k(x_j)}.$$

Observe that if the same face $(ikj)$ (with opposite orientation)  is common to a second tetrahedron then $$z_{ikj} = \frac{1}{z_{ijk}}.$$

Figure \ref{tetra2} displays the coordinates.

\begin{figure}[ht]      
\begin{center}
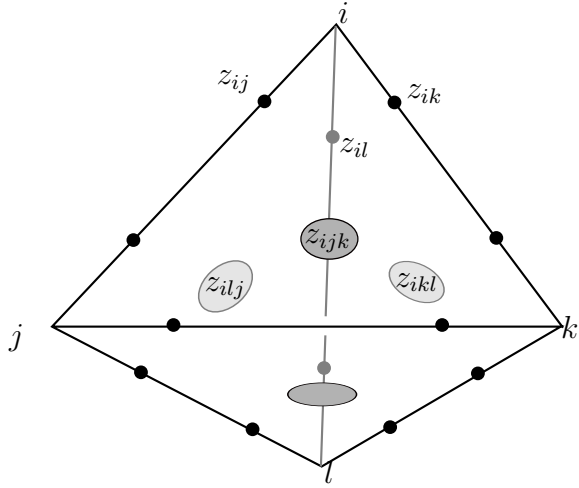
\caption{The $z$-coordinates for a tetrahedron} \label{tetra2}
\end{center}
\end{figure}

\subsection{} Of course there are relations between the whole set of coordinates. Fix an even permutation $(i,j,k,l)$ of $(1,2,3,4)$. First, for each face $(ijk)$, the $3$-ratio is the opposite of the product of all cross-ratios ``leaving'' this face:
\begin{equation} \label{Rel1}
z_{ijk}= - z_{il}z_{jl}z_{kl}.
\end{equation}
Second, the three cross-ratio leaving a vertex are algebraically related:
\begin{equation} \label{Rel2}
\begin{split}
& z_{ik} =  \frac{1}{1-z_{ij}}\\
& z_{il} =   1-\frac{1}{z_{ij}}
\end{split}
\end{equation}

 Relations \ref{Rel2}  are directly deduced from the definition of the coordinates $z_{ij}$, while relation \ref{Rel1} is a direct consequence of lemma \ref{fact1}.

At this point, we choose four coordinates, one for each vertex: $z_{12}$, $z_{21}$, $z_{34}$, $z_{43}$. The next proposition shows that a tetrahedron is uniquely determined by these four numbers, up to the action of $\PGL(3)$.  It also 
shows that the space of {\sl cross-ratio structures}
 on a tetrahedron defined
in \cite{falbel} coincides with the space of generic tetrahedra as defined above.

\begin{proposition}\label{parameters}
A tetrahedron of flags is parametrized by the $4$-tuple
$(z_{12},z_{21},z_{34},z_{43})$ of elements in $k\setminus \{0, \, 1\}$. 
\end{proposition}

\begin{proof} 
 Let $e_1,\, e_2, \, e_3$ be the canonical basis of $V$ and $(e_1^*,e_2^*,e_3^*)$ its dual basis.
Up to the action of $\SL(3)$, an element $([x_i],[f_i])$ of $\C^r_3(\Fl)$ is uniquely given, in these coordinates, by:
\begin{itemize}
 \item $x_1=(1,0,0)$, $f_1=(0,z_1,-1)$,
 \item $x_2=(0,1,0)$, $f_2=(z_2,0,-1)$,
 \item $x_3=(0,0,1)$, $f_3=(z_3,-1,0)$ and
 \item $x_4=(1,1,1)$, $f_4=z_4(1,-1,0)+(0,1,-1)$.
\end{itemize}
Observe that  $z_i\neq 0$ and $z_i\neq 1$ by the genericity condition.  
Now we compute, using lemma \ref{fact1} for instance, that  
 $z_{12}= \frac{1}{1-z_1}$, $z_{21}=1-z_2$, $z_{34}=z_3$, $z_{43}=1-z_4 $, completing the proof.
\end{proof}

We note that 
one can then compute $\bar{d}_3 : C^r_3(\Fl)_G\rightarrow
C^r_{2}(\Fl)_G$ on the generators of $C^r_3(\Fl)_G$ to be
$$
\bar{d}_3(z_{12},z_{21},z_{34},z_{43})=[z_{123}]-[z_{124}]+[z_{134}]-[z_{234}].
$$
\subsection{Coordinates for affine flags}\label{ss:coord-uni}

We will also need coordinates for a tetrahedron of affine flags 
(the $\mathcal A$-coordinates in Fock and Goncharov \cite{FG}). 
Let $(x_i,f_i)_{1\leq i \leq 4}$ be an element of $C_3(\AFl)$. We also define a set of 12 coordinates on the edges of the thetrahedron 
(one for each oriented edge) and four coordinates associated to the faces:

We associate to the edge $ij$ the number $a_{ij}=f_i(x_j)$  
and to the face $ijk$ (oriented as the boundary of the tetrahedron) the number $a_{ijk}=\det(x_i,x_j,x_k)$.

We remark that for a tetrahedron of affine flags, the $z$-coordinates are well-defined, and are ratios of the affine coordinates:
\begin{equation}
z_{ij} = \frac{a_{ik} a_{ijl}}{a_{il} a_{ijk}} \ \ \ \mbox{ and } \ \ \ z_{ijk} = \frac{a_{ij} a_{jk} a_{ki}}{a_{ik} a_{ji} a_{kj}} .
\end{equation}

\section{Tetrahedra of flags and volume} \label{bloch}

In this section we define the \emph{volume} of a tetrahedron of flags, 
generalizing and unifying the volume of hyperbolic tetrahedra 
(see section \ref{hyp}) and CR tetrahedra 
(see \cite{falbel} and section \ref{CR}). Via proposition \ref{parameters}, it coincides with the 
volume function on 
cross-ratio structures on a tetrahedron as defined in \cite{falbel}. 
 We then define the volume of a simplicial complex of flags tetrahedra. This volume is invariant under a change of triangulation of the simplicial complex (2-3 move) hence is naturally an element of the pre-Bloch group  and the volume is defined on the third homology group of 
flag  configurations (see also \cite{FalbelWang}).  Eventually we get a  map, still called volume, from  the third (discrete) homology group of $\SL(3)$ to the Bloch group, through the natural projection from $H_3(\SL(3))$ to $H_3(\Fl)$. We conclude the section with the proof that this last map actually coincides with the Suslin map from $H_3(\SL(3))$ to the Bloch group.

\subsection{The pre-Bloch and Bloch groups, the dilogarithm}

We define a volume for a tetrahedron of flags by constructing an element of the pre-Bloch group and then taking the dilogarithm map.


The {\it pre-Bloch group} $\mathcal{P}(k)$ is the quotient of the free
abelian group $\bZ[k\setminus\{0,1\}]$ by the subgroup generated by
the 5-term relations
\begin{equation}\label{5term}
 [x]-[y]+[\frac{y}{x}]-[\frac{1-x^{-1}}{1-y^{-1}}]+[
  \frac{1-x}{1-y}],\ \  \forall\, x,y\in k\setminus\{0,1\}.
\end{equation}

For a tetrahedron of flags $T$, let $z_{ij} = z_{ij}(T)$ and $z_{ijk} = z_{ijk}(T)$ be its coordinates. 

\subsection{}
To each tetrahedron 
define
the element
$$
\beta(T)= [z_{12}]+[z_{21}]+[z_{34}]+[z_{43}]\in {\mathcal P(\bC)}
$$
and extend it -- by linearity -- to a function
$$\beta \; : \; \C_3(\Fl(k)) \to {\mathcal P(\bC)}.$$

We emphasize here that $\beta (T)$ depends on the ordering of the vertices of each tetrahedron\footnote{This assumption may be removed by averaging $\beta$ over all orderings of the vertices. In 
any case if $c$ is a chain in $C_3 (\Fl (k))$ representing a cycle in $\C_3 (\Fl (k))$ we can represent
$c$ by a closed $3$-cycle $K$ together with a numbering of the vertices of each tetrahedron of $K$ (see section \ref{ss:neumanncomplex}).}
$T$. The following proposition implies that $\beta$ is well defined on $H_3(\Fl)$. 

\begin{proposition}\label{betaH3}
We have: $\beta (\overline{d}_4 (\C_4(\Fl))) =0$.
\end{proposition}
\begin{proof} We have to show that $\mathrm{Im} (\overline{d}_4 )$ is contained in the subgroup generated by the $5$-term relations. This is proven by computation and is exactly the content of \cite[Theorem 5.2]{falbel}.
\end{proof}

\subsection{} We use wedge $\wedge_{_\mathbb{Z}}$ for skew symmetric product on Abelian groups. 
Consider $k^{\times}\wedge_{_\bZ} k^{\times}$, where $k^{\times}$ is the
multiplicative group of $k$. It is  the abelian group generated by the set $x\wedge_{_\mathbb{Z}} y$ factored by the relations 
$$xy \wedge_{_\mathbb{Z}} z = x \wedge_{_\mathbb{Z}} z + y \wedge_{_\mathbb{Z}} z \mbox{ and } 
x\wedge_{_\mathbb{Z}} y = - y\wedge_{_\mathbb{Z}} x.$$
In particular, $1 \wedge_{_\mathbb{Z}} x=0$ for any $x \in  k^{\times}$, and 
$$x^n \wedge_{_\mathbb{Z}} y = n (x\wedge_{_\mathbb{Z}} y) = x\wedge_{_\mathbb{Z}} y^n.$$

\subsection{}
The {\it Bloch group} $\mathcal{B}(k)$ is the kernel of the homomorphism
$$
\delta : \mathcal{P}(k) \rightarrow k^{\times}\wedge_{_\bZ} k^{\times},
$$
which is defined on  generators of $\mathcal{P}(k)$  by
$\delta ([z])=z\wedge(1-z)$.

The {\it Bloch-Wigner dilogarithm} function is
\begin{align*}
D(x) & = \arg{(1-x)}\log{|x|}-\mathrm{Im} (\int_{0}^{x}\log{(1-t)}\frac{dt}{t}), \\
& = \arg{(1-x)}\log{|x|} + \mathrm{Im}( \ln_2 (x)).
\end{align*}
Here  $ \ln_2 (x)=\int_{0}^{x}\log{(1-t)}\frac{dt}{t}$ is the dilogarithm function.
The function $D$ is well-defined and real analytic on $\bC-\{0,1\}$ and
extends to a continuous function on $\bC P^1$ by
defining $D(0) = D(1) = D(\infty) = 0$. It satisfies the 5-term relation and therefore,
for $k$ a subfield of $\bC$, gives rise
to a well-defined map:
$$
  D: \mathcal{P}(k)\rightarrow \bR,
$$
given by linear extension as
$$
 D(\sum_{i=1}^{k}n_i[x_i])=\sum_{i=1}^{k}n_iD(x_i).
$$

\subsection{} \label{par:vol} We finally define the {\it volume map} on $\C_3(\Fl)$ via the dilogarithm (the constant will be explained in the next section):
$$\Vol \, = \frac{1}{4}  D \circ \beta  \; : \; \C_3(\Fl(k)) \to \bC.$$

From Proposition \ref{betaH3}, $\Vol$ is well defined on $H_3(\Fl)$.

\subsection{The hyperbolic case}\label{hyp}

We briefly explain here how the hyperbolic volume for ideal tetrahedra in the hyperbolic space $\mathbb H^3$ fits into the framework described above.

An ideal hyperbolic tetrahedron is given by $4$ points on the boundary of $\mathbb H^3$, 
i.e. $\P_1(\bC)$. Up to the action of $\SL(2,\bC)$, these points are in homogeneous coordinates
 $[0,1]$, $[1,0]$, $[1,1]$ and $[1,t]$ -- the complex number $t$ being the cross-ratio of these four points. 
So its volume is $D(t)$ (see e.g. \cite{Zagier}).

Identifying $\bC^3$ with the Lie algebra $\mathfrak{sl}(2,\bC)$, we have the adjoint action 
of $\SL(2,\bC)$ on $\bC^3$ preserving the quatratic form given by the determinant, 
given in canonical coordinates by $xz-y^2$. The group $\SL(2,\bC)$ preserves the 
isotropic cone of this form. The projectivization of this cone is identified to $\P_1(\bC)$ 
via the Veronese map (in canonical coordinates):
\begin{eqnarray*}
 h_1 \: : \:  \P_1(\bC) & \to & \P_2(\bC) \\
 \left[x,y\right] & \mapsto &  [x^2,xy,y^2]
\end{eqnarray*}
The first jet of that map gives a map $h$ from $\P_1(\bC)$ to the variety of flags $\Fl$.
A convenient description of that map is obtained thanks to the 
identification between $\bC^3$ and its dual given by the quadratic form. Denote $\langle ,\rangle$ 
the bilinear form associated to the determinant. Then we have
\begin{eqnarray*}
 h \: : \:\P_1(\bC) & \to & \Fl(\bC)\\
 p & \mapsto &  \left(h_1(p), \langle h_1(p), \cdot \rangle \right).
\end{eqnarray*}

Let $T$ be the tetrahedron $h([0,1])$, $h([1,0])$, $h([1,1])$ and $h([1,t])$. An easy computation 
gives its coordinates:
$$ z_{12}(T) = t  \hskip1cm z_{21}(T)=t \hskip1cm z_{34}(T)=t \hskip1cm z_{43}(T)=t.$$

It implies that $\beta(T)=4t$ and our function $\Vol$ coincide with the hyperbolic volume: 
$$\Vol(T)=D(t).$$

\begin{remark}
{\rm Define an involution $\sigma$ on the $z$-coordinates by:
$$\sigma (z_{ijk} ) = \frac{1}{z_{ijk}}$$
on the faces and
$$\sigma (z_{ij}) = \frac{z_{ji} (1+z_{ilj})}{z_{ilj} (1+ z_{ijk})} \mbox{ and } 
\sigma (z_{ji} )= \frac{z_{ij} (1+z_{ijk})}{z_{ijk} (1+z_{ilj})}$$
on edges. The set of fixed points of $\sigma$ correspond exactly with the hyperbolic tetrahedra.}
\end{remark}

\subsection{The CR case}\label{CR}

CR geometry is modeled on the sphere ${\mathbb S}^3$ equipped with a natural $\mathrm{PU}(2,1)$ action.  More precisely, consider the group $\mathrm{U}(2,1)$ preserving the Hermitian form
$\langle z,w \rangle = w^*Jz$ defined on ${\bC}^{3}$ by
the matrix
$$
J=\left ( \begin{array}{ccc}

                        0      &  0    &       1 \\

                        0       &  1    &       0\\

                        1       &  0    &       0

                \end{array} \right )
$$
and the following cones in ${\mathbb C}^{3}$;
$$
        V_0 = \left\{ z\in {\bC}^{3}-\{0\}\ \ :\ \
 \langle z,z\rangle = 0 \ \right\},
$$
$$
        V_-   = \left\{ z\in {\bC}^{3}\ \ :\ \ \langle z,z\rangle < 0
\ \right\}.
$$
Let $\pi :{\bC}^{3}\setminus\{ 0\} \rightarrow \mathbb{CP}^{2}$ be the
canonical projection.  Then
${\mathbb H}_{\bC}^{2} = \pi(V_-)$ is the complex hyperbolic space and its boundary is
$$
\partial{\mathbb H}_{\bC}^{2} ={\mathbb S}^3= \pi(V_0)=\{ [x,y,z]\in \mathbb{CP}^{2}\ |\ x\bar z+ |y|^2+z\bar x=0\ \}.
$$
The group of biholomorphic transformations of ${\mathbb H}_{\bC}^{2}$ is then
$\mathrm{PU}(2,1)$, the projectivization of $\mathrm{U}(2,1)$.  It acts on ${\mathbb S}^3$ by
CR transformations.  

An element $x\in {\mathbb S}^3$ gives rise to an element  $([x],[f])\in \Fl(\bC)$ where $[f]$ corresponds to the unique complex line tangent to ${\mathbb S}^3$ at $x$.  
As in the hyperbolic case we may consider the inclusion map
\begin{eqnarray*}
 h_1 \: : \:  {\mathbb S}^3 & \to & \P_2(\bC) \\
\end{eqnarray*}
and the first complex jet of that map gives a map
\begin{eqnarray*}
 h \: : \: {\mathbb S}^3 & \to & \Fl(\bC)\\
 x & \mapsto &  \left(h_1(x), \langle . , h_1(x) \rangle \right)
\end{eqnarray*}
A generic configuration of four points in ${\mathbb S}^3$ is given, up to $\mathrm{PU}(2,1)$, by the following
four elements in homogeneous coordinates (we also give for each point $x_i$, the corresponding 
dual $f_i$ to the complex line containing it and tangent to ${\mathbb S}^3$):
\begin{itemize}
 \item $x_1=(1,0,0)$, $f_1=(0,0,1)$,
 \item $x_2=(0,0,1)$, $f_2=(1,0,0)$,
 \item $x_3=(\frac{-1+it}{2},1,1)$, $f_3=(1,1,\frac{-1-it}{2})$ and
 \item $x_4=(\frac{|z|^2(-1+is)}{2},z,1)$, $f_4=(1,\bar z,\frac{-|z|^2(1+is)}{2})$
\end{itemize}
with $z\neq 0,1$ and  $\overline{z}\frac{s+i}{t+i}\neq 1$.  Observe that $\mathrm{PU}(2,1)$ acts doubly transitively on ${\mathbb S}^3$ and 
for a generic triple of points $x_1,x_2,x_3$ the triple ratio of the corresponding flags
 is given by $\frac{1-it}{1+it}$.
 One can easily compute the invariants of the tetrahedron:
$$
 z_{12}=z,\; z_{21}=\frac{\overline{z}(s+i)}{t+i},\; z_{34}=\frac{z[(t+i)-\overline{z}(s+i)]}{(z-1)(t-i)},\;
 z_{43}=\frac{\overline{z}(z-1)(s-i)}{(t+i)-\overline{z}(s+i)}.
$$
The following proposition describes the space of
generic configurations of four points in ${\mathbb S}^3$.
\begin{proposition} 
Configurations (up to translations by $\mathrm{PU}(2,1)$) of four generic points
in ${\mathbb S}^3$ are parametrised
by elements in  $\C^r_3(\Fl)$
with coordinates $z_{ij}$,  $1\leq i\neq j\leq 4$ satisfying  the three complex equations
\begin{eqnarray}\label{eq:cr}
z_{ij}z_{ji}=\overline {z_{kl}z_{lk}}
\end{eqnarray}
with the exclusion of solutions such that
$z_{ij}z_{ji}z_{ik}z_{ki}z_{il}z_{li}=-1$ and $z_{ij}z_{ji}\in \bR$.
\end{proposition}

 As in \cite{falbel} (up to multiplication by 4)  the volume of 
a CR tetrahedron $T_{CR}$  is
$\Vol(T_{CR})=\frac{1}{4}D\circ \beta(T_{CR})$.

\subsection{Relations with the work of Suslin}

We show here how our map $\beta$ allows a new and more geometric way to interpret Suslin map $S : H_3(\SL(3)) \to \pB$ (see \cite{suslin}). First of all, recall that the natural projection $\pi \: : \: \SL(3) \to \Fl=\PGL(3)/B$ gives a map $\pi_* \: : \: H_3(\SL(3)) \to H_3(\Fl)$.

\begin{theorem}
 We have $\beta \circ \pi_* = 4S$.
\end{theorem}

\begin{proof}
Let $T$ be the subgroup of diagonal matrices (in the canonical basis) of $\SL(3)$. Recall that $\SL(2)$ is seen as a subgroup of $\SL(3)$ via the adjoint representation (as in section \ref{hyp}).
 We find in the work of Suslin the three following results:
\begin{enumerate}
 \item $H_3(\SL(3))=H_3(\SL(2))+ H_3(T)$ \cite[p. 227]{suslin}
 \item $S$ vanishes on $H_3(T)$ \cite[p. 227]{suslin}
 \item $S$ coincide with the cross-ratio on $H_3(\SL(2))$ \cite[lemma 3.4]{suslin}.
\end{enumerate}

So we just have to understand the map $\beta\circ \pi_*$ on $T$ and $\SL(2)$. As $T$ is a subgroup of $B$, the map $\beta \circ \pi_*$ vanishes on $T$. And we have seen in the section \ref{hyp} that, on a hyperbolic tetrahedron, $\beta$ coincide with $4$ times the cross-ratio.

This proves the theorem.
\end{proof}

\medskip
\noindent
{\it Remark.} After writing this section we became aware of Zickert's paper \cite{Zickert}. In it
(see \S 7.1)
Zickert defines a generalization -- denoted $\widehat{\lambda}$ -- of Suslin's map. When specialized to
our case his definition coincides with $\frac{1}{4} \beta \circ \pi^*$. We believe that the construction
above sheds some light on the ``naturality'' of this map. 

\section{Decoration of a tetrahedron and the pre-Bloch group}\label{decoration}

In this section we let $T$ be an ordered tetrahedron of flags and compute in two different ways $\delta(\beta (T))$.
The first -- and most natural -- way uses $a$-coordinates associated to some lifting of $T$ as
a tetrahedron of affine flags. In that respect we mainly follow Fock and Goncharov. The second way directly deals with $z$-coordinates and follows the approach of Neumann and Zagier. Finally we explain
how the two ways are related; we will see in the remaining of the paper how fruitful it is to mix them.
 
\subsection{Affine decorations and the pre-Bloch group} \label{acoord}

We first let $(x_i ,f_i)_{1\leq i \leq 4}$ be an element of $\C_3(\AFl)$ lifting $T$. This allows us to associate
$a$-coordinates to $T$. 

Let $J_T^2 = \mathbb{Z}^I$ be the $16$-dimensional abstract free $\mathbb{Z}$-module where (see figure \ref{FG})
$$I = \{ \mbox{vertices of the (red) arrows in the $2$-triangulation of the faces of $T$} \}.$$
We denote the canonical basis $\{e_{\alpha} \}_{\alpha \in I}$ of $J_T^2$.  It contains
 {\it oriented edges} $e_{ij}$ (edges oriented from $j$ to $i$) and {\it faces} $e_{ijk}$. Given $\alpha$ and $\beta$ in $I$ we set:
$$\varepsilon_{\alpha \beta} = \# \{ \mbox{oriented (red) arrows from $\alpha$ to $\beta$}\} -  \# \{ \mbox{oriented (red) arrows from $\beta$ to $\alpha$}\}.$$

\begin{figure}[ht]      
\begin{center}
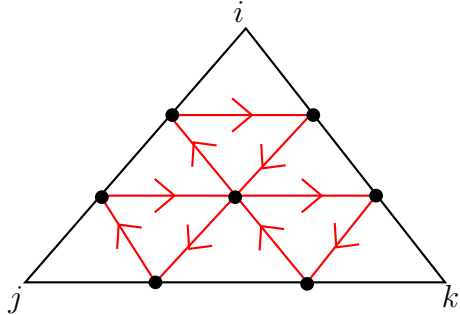
\caption{Combinatorics of $W$} \label{FG}
\end{center}
\end{figure}

\subsection{} \label{R} The $a$-coordinates $\{ a_{\alpha} \}_{\alpha \in I}$ of our tetrahedron of affine flags $T$ now define an element $ \sum_{\alpha \in I} a_{\alpha} e_{\alpha}$ of
$k^{\times} \otimes_{\mathbb{Z}} J_T^2 \cong \mathrm{Hom} ((J_T^2)^* , k^{\times})$ where $k$ is any field which contains all the $a$-coordinates.

Let $V$ be a $\bZ$-module equipped with a bilinear product 
$$
B: V\times V\rightarrow \bZ.
$$ 
We consider on the $k^\times$-module
$V_k = k^{\times}\otimes_\bZ V$ the bilinear product
$$
{\wedge_B} : V_k\times V_k\rightarrow k^{\times}\wedge_\bZ k^{\times}
$$
defined on generators by
$$
(z_1\otimes v_1)\wedge (z_2\otimes v_2)=B(v_1,v_2) (z_1\wedge z_2).
$$



In particular letting $\Omega^2$ be the bilinear skew-symmetric form on $J_T^2$ given by\footnote{Observe in particular that 
$\Omega^2(e_{ji} , e_{ijk})= 1$
and so on, 
the logic being that the vector $e_{ijk}$ is the outgoing vector on the face $ijk$ and the vector
 $e_{ji}$ (oriented from $i$ to $j$) turns around it in the positive sense. }
 $$\Omega^2 ( e_{\alpha} , e_{\beta} ) = \varepsilon_{\alpha \beta},$$
we get:
$$a \wedge_{\Omega^2} a = \sum_{\alpha , \beta \in I} \varepsilon_{\alpha \beta} a_{\alpha} \wedge_{\bZ} a_{\beta}.$$ 

\begin{lemma} \label{deltaobeta}
We have:
\begin{equation} \label{deltaobeta1}
\delta (\beta (T)) = \frac12 a \wedge_{\Omega^2} a.
\end{equation}
\end{lemma}
\begin{proof} 
To each ordered face $(ijk)$ of $T$ we associate the element
\begin{equation} \label{Wface}
W_{ijk} = a_{ijk} \wedge \frac{a_{ki} a_{jk} a_{ij}}{a_{ik} a_{kj} a_{ji}} 
+ a_{ij} \wedge a_{ik} + a_{ki} \wedge a_{kj} + a_{jk} \wedge a_{ji} \in k^{\times} \wedge_{\bZ} k^{\times}.
\end{equation}
The proof in the CR case of \cite[Lemma 4.9]{FalbelWang}  obviously
leads to\footnote{Alternatively we may think of $T$ as a geometric realization
of a mutation between two triangulations of the quadrilateral $(1324)$ and apply \cite[Corollary 6.15]{FG}.}:
$$\delta (\beta (T)) = W_{143} + W_{234} + W_{132} + W_{124}.$$
Finally one easily sees that 
$$W_{143} + W_{234} + W_{132} + W_{124} = \frac12  \sum_{\alpha , \beta \in I} \varepsilon_{\alpha \beta} a_{\alpha} \wedge_{\bZ} a_{\beta}.$$ 
\end{proof}

We let 
$$W(T) = W_{143} + W_{234} + W_{132} + W_{124}.$$

\medskip
\noindent
{\it Remark.} 1. The element $W(T)$ coincides with the $W$ invariant associated by Fock and Goncharov to the triangulation by a tetrahedron of a sphere with $4$ punctures. (The 
orientation of the faces being induced by the orientation of the sphere.)

2. Whereas $T$ -- being a tetrahedron of flags -- only depends on the flag
 coordinates, 
each $W$ associated to the faces depends on the {\it affine} flag coordinates.

\medskip

In the next paragraph we make remark 2 more explicit by computing 
$\delta (\beta (T))$ using the $z$-coordinates.

\subsection{The Neumann-Zagier symplectic space}\label{ss:NZsympl}

In this section we  analyse an extension of Neumann-Zagier symplectic space introduced by J. Genzmer \cite{Genzmer}
in  the space of $z$-coordinates associated to the edges of a tetrahedron.  We reinterpret her definitions in our context
of flag tetrahedra.
Recall that we have associated $z$-coordinates to a tetrahedron of flags $T$. These consists of 
$12$ edge coordinates $\{ z_{ij} \}$ and $4$ face coordinates $\{ z_{ijk} \}$ subject to the 
relations \eqref{Rel1} and \eqref{Rel2}. Recall that relation  \eqref{Rel1} is $z_{ijk}=-z_{il}z_{jl}z_{kl}$ and note that \eqref{Rel2} implies in particular that:
\begin{equation} \label{Rel3}
z_{ij} z_{ik} z_{il} = -1 .
\end{equation}
We linearize \eqref{Rel1} and \eqref{Rel3} in the following way: 
We let $J_T$ be the $\bZ$-module obtained as the quotient of $J_T^2 = \bZ^I$ by the kernel of $\Omega^2$. The latter is the subspace generated by elements of the form 
$$\sum_{\alpha \in I} b_{\alpha} e_{\alpha}$$
for all $\{ b_{\alpha} \} \in \mathbb{Z}^I$ such that $\sum_{\alpha \in I} b_{\alpha} \varepsilon_{\alpha \beta} =0$ for every $\beta \in I$. Equivalently it is the subspace generated by $e_{ij}+e_{ik}+e_{il}$ and $e_{ijk}- (e_{il}+e_{jl}+e_{kl})$. We will rather use as generators the elements
$$v_i = e_{ij}+e_{ik}+e_{il} \mbox{  and  } w_i = e_{ji} + e_{ki} + e_{li} + e_{ijk} + e_{ilj} + e_{ikl},$$
see Figure \ref{fig:v_i-w_i0}.  

\begin{figure}[ht]      
\begin{center}
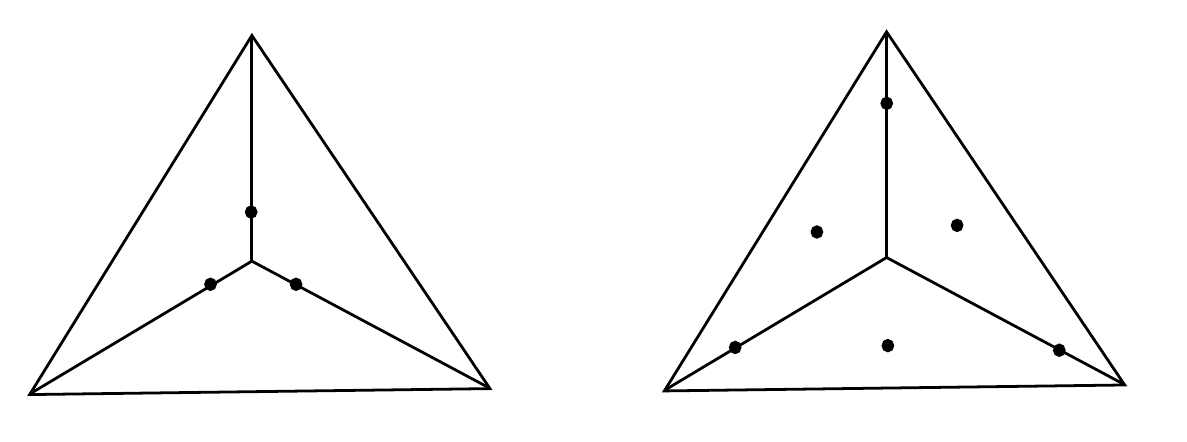
\caption{The vectors $v_i$ and $w_i$ in $\mathrm{Ker}(p)$} \label{fig:v_i-w_i0}
\end{center}
\end{figure}

We let $J_T^* \subset (J_T^2)^*$ be the dual subspace which consists of the linear maps in $(J_T^2)^*$ which vanish on the kernel of $\Omega^2$. Note that $J_T$ (as well as $J_T^*$) is $8$-dimensional.

\subsection{} The $z$-coordinates $\{ z_{\alpha} \}_{\alpha \in I}$ of our tetrahedron of flags $T$ now define an element 
$$z= \sum_{\alpha \in I} z_{\alpha} e_{\alpha}^*\in \mathrm{Hom} (J_T^2 , k^{\times}) \cong
k^{\times} \otimes_{\mathbb{Z}} (J_T^2)^*$$ 
where $k$ is any field which contains the $z$-coordinates.
Because of \eqref{Rel1} and \eqref{Rel3} the image of the kernel of $\Omega^2$ by $z$ is the (torsion) subgroup $\{ \pm 1 \} \subset k^{\times}$ (that is easily checked on $v_i$ and $w_i$). Denoting $V\left[ \frac12 \right]$ the tensor product $V \otimes_{\bZ} \bZ 
\left[ \frac12 \right]$ of a $\Z$-module $V$, we conclude that the element $z \in k^{\times} \otimes (J_T^2)^* \left[ \frac12 \right]$ in fact belongs to $k^{\times} \otimes J_T^* \left[ \frac12 \right]$. 

\subsection{} The space $J_T^*$ is $8$-dimensional and we may associate to $8$ oriented edges (two pointing at each vertex)
of $T$ a basis $\{ f_{ij} \}$. Using this basis, the element $z \in  k^{\times} \otimes_{\bZ} J_T^*\left[ \frac12 \right]$ is written
$z = \sum z_{ij} f_{ij}$. 

We then note that (up to eventually adding a root of $-1$ to $k$):
\begin{equation} \label{deltabetaT}
\begin{split}
\delta (\beta (T)) & =  z_{ij} \wedge_{\bZ} (1-z_{ij}) + z_{ji} \wedge_{\bZ}  (1-z_{ji})  \\
& \ \ \ \ \ \  + z_{kl} \wedge_{\bZ} (1-z_{kl}) + z_{lk} \wedge_{\bZ} (1-z_{lk})  \\
& =  \frac12 z \wedge_{\mathbb{Z}} H z ,
\end{split}
\end{equation}
where $H$ is the linear map $J_T^* \rightarrow J_T^*$ which on generators of $J_T^*$ is given by 
$H(f_{ij})=f_{ ik}$ and $H(f_{ik})=-f_{ ij}$. It yields a linear map 
$H : k^{\times} \otimes_{\mathbb{Z}} J_T^*  \rightarrow k^{\times} \otimes_{\mathbb{Z}} J_T^*$. 
We note that in coordinates:
$$
(Hz)_{f_{ij}}=\frac{1}{z_{ik}}\ \  \mbox{and}\ \ (Hz)_{f_{ik}}=z_{ij}.
$$

\subsection{} The choice of the basis $\{ f_{ij} \}$ of $J_T^*$ and the choice of the map $H$ are 
not canonical but they define a natural symplectic form 
\begin{equation} \label{Omega}
\Omega^* ( \cdot , \cdot )= \langle H \cdot , \cdot \rangle
\end{equation}
on $J_T^*$ where $\langle , \rangle$ is the scalar product associated to the basis $\{f_{ij} \}$. 
Such a symplectic space was first considered by Neumann and Zagier (see \cite{NeumanZagier,Neuman}) in the $\SL(2 , \bC)$ context. 

The following lemma now follows from \eqref{deltabetaT} and \eqref{Omega}.

\begin{lemma} \label{deltaobeta'}
We have:
\begin{equation} \label{deltaobeta2}
\delta (\beta (T)) = \frac12 z \wedge_{\Omega^*} z .
\end{equation}
\end{lemma}

\subsection{Relation between $a$ and $z$-coordinates}
Let 
$$p : J_T^2 \rightarrow (J_T^2)^*$$
be the homomorphism $v \mapsto \Omega^2 (v , \cdot )$. On the basis $(e_{\alpha})$ and its dual $(e_{\alpha}^*)$, we can write
$$p (e_{\alpha}) = \sum_{\beta} \varepsilon_{\alpha \beta} e_{\beta}^*.$$
We define accordingly the dual map
$$p^* : \mathrm{Hom}((J_T^2)^*,k^{\times}) \rightarrow \mathrm{Hom}(J_T^2,k^{\times}).$$
Observe that if $a \in k^{\times} \otimes_{\bZ} J_T^2$ and $z \in k^{\times} \otimes_{\bZ} (J_T^2)^*$ are the elements associated to the $a$ and $z$-coordinates of $T$ then:
$$p^* (a) = z \quad \mbox{in } k^{\times} \otimes J_T^* \left[ \frac12 \right].$$
Indeed,
$$p^* (a)(e_{\alpha}) = a(p(e_{\alpha})) = a (\sum_{\beta} \epsilon_{\alpha  \beta} e_{\beta}^*)
= \prod_{\beta} a_{\beta}^{\varepsilon_{\alpha \beta}}.
$$
In particular, we recuperate the formula
$$
z_{ij} = \frac{a_{ik} a_{ijl}}{a_{il} a_{ijk}} .
$$
Note however that in our conventions the coordinate $a_{ijl}$ should be written $a_{ilj}$. There is therefore a sign missing here and $p^*(a)=z$ only holds modulo $2$-torsion.
  
The image $p(J_T^2) \subset (J_T^2)^*$ coincides with $J_T^*$ and one easily checks 
that $p^* (\Omega^* ) = \Omega^2$. It
then follows from the following lemma that
\begin{equation} \label{az}
a \wedge_{\Omega^2} a = z\wedge_{\Omega^*} z
\end{equation} 
which explains the coincidence of lemma \ref{deltaobeta} and lemma \ref{deltaobeta'}.

\begin{lemma} \label{lemfacile} 
If $\phi: V\rightarrow W$ is a homomorphism of $\bZ$-modules equipped with bilinear forms $B$ and $b$
such that $\phi^*(b)=B$ then the induced map
$$
\phi: V_k\rightarrow  W_k
$$
satisfies
$$
\phi^*({\wedge_b})={\wedge_B}
$$
\end{lemma}
\begin{proof} This is a simple consequence of the definitions.
\end{proof}

\subsection{} Note that the form $\Omega^2$ induces a -- now non-degenerate -- symplectic form $\Omega$ on $J_T$. This yields a canonical identification between $J_T$ and $J_T^*$; the form $\Omega^*$ is the corresponding symplectic form. We may therefore as well work with $(J_T , \Omega)$ as with $(J_T^* , \Omega^*)$. The bilinear form $\Omega$ on $J_T$ is characterized as 
the non-singular skew-symmetric form given by
$$\Omega (e_\alpha , e_\beta ) = \varepsilon_{\alpha\beta}.$$

\section{Decoration of a tetrahedra complex and its holonomy}\label{decoration2}

In the previous sections we defined coordinates for a tetrahedron of flags and affine flags and defined its volume in $\mathcal{P} (\bC)$. We study here how one may decorate a complex of tetrahedra with these coordinates, compute the holonomy of its fondamental group. We also investigate the invariant $\beta$ (in the pre-Bloch group) of the decorated complex. We eventually state the main theorem of the paper, theorem \ref{thm:kabaya}, which computes $\delta(\beta)$ in terms of the holonomy.

\subsection{Quasi-simplicial complex and its decorations}

Let us begin with the definition of a quasi-simplicial complex (see e.g. \cite{Neuman}):
A {\it quasi-simplicial complex} $K$ is a cell complex whose cells are simplices with 
injective simplicial attaching maps, but no requirement that closed simplices embed in $|K|$ -- the 
underlying topological space. 
A {\it tetrahedra complex} is a quasi-simplicial complex of dimension $3$.

\subsection{} From now on we let $K$ be a tetrahedra complex.
The {\it (open) star} of a vertex $v \in K^{(0)}$ is the union of all
the open simplices that have $v$ as a vertex. It is an open neighborhood of $v$ and is the open 
cone on a simplicial complex $L_v$ called the {\it link} of $v$.

A {\it quasi-simplicial $3$-manifold} is a compact tetrahedra complex $K$ such
that $|K|-|K^{(0)}|$ is a $3$-manifold (with boundary). By an orientation of $K$ we mean an orientation
of this manifold. 
A {\it $3$-cycle} is a closed quasi-simplicial $3$-manifold.

\subsection{}A quasi-simplicial $3$-manifold is topologically a manifold except perhaps 
for the finitely many singular points $v \in |K^{(0)}|$ where the local structure is that of a cone on 
$|L_v|$ -- a compact connected surface (with boundary). We will soon require that for each 
vertex $v \in K^{(0)}$, $|L_v|$ is homeomorphic to either a sphere, a torus or an annulus. Let $K_s^{(0)}$, $K_t^{(0)}$ and 
$K_a^{(0)}$ the corresponding subsets of vertices. We note that $|K|-|K_t^{(0)} \cup K_a^{(0)}|$ 
is an (open) $3$-manifold with boundary that retracts onto a compact $3$-manifold with boundary $M$.
Note that $\partial M$ is the disjoint union $T_1 \cup \ldots  \cup T_{\tau} \cup S_1 \cup \ldots \cup S_{\sigma}$ where each $T_i$ is a torus and each $S_i$ a surface of genus $g_i \geq 2$. Moreover:
each $T_i$ corresponds to a vertex in $K_t^{(0)}$ and each $S_i$ contains at least one simple 
closed essential curve each corresponding to a vertex in $K_a^{(0)}$, see figure \ref{Trig}. 

Given such a compact oriented $3$-manifold with boundary $M$. We call a quasi-simplicial $3$-manifold as above a {\it triangulation} of $M$. 

\begin{figure}[ht]      
\begin{center}
\includegraphics[scale=.3]{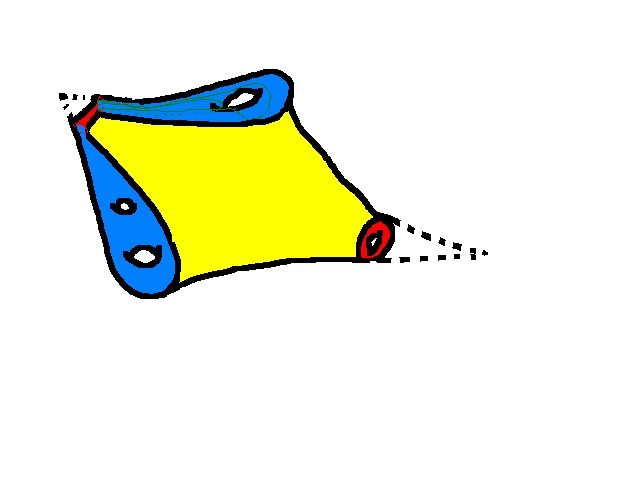}
\caption{The retraction of a quasi-simplicial $3$-manifold onto a compact $3$-manifold with boundary} \label{Trig}
\end{center}
\end{figure}

A {\it decoration} of a tetrahedra complex is an incarnation of this complex in our spaces of flags 
or affine flags: 

\subsection{}
A {\it parabolic decoration} of the tetrahedra complex is the data of a flag for each vertex (equivalently a map from the $0$-skeleton of the complex to $\Fl$) such that, for each tetrahedron of the complex, the corresponding tetrahedron of flags is in generic position.
Similarily, a {\it unipotent decoration} is the data of an affine flag for each vertex with the genericity condition.

Let us make two comments on these definitions. First, any parabolic decoration -- 
together with an ordering of the vertices of each $3$-simplex -- equip each tetrahedron with a set of coordinates as defined in section \ref{ss:coord-para}. Second, an unipotent decoration
induces a parabolic decoration via the canonical projection $\AFl\to \Fl$, so we get these coordinates, as well as a set of affine coordinates (see section \ref{ss:coord-uni}).

\subsection{}\label{ss:neumanncomplex}
Neumann \cite[\S 4]{NeumannGT} has proven that any element of $H_3 (\PGL_3 (\mathbb{C}))$ can be represented by an oriented $3$-cycle $K$ together with an ordering of the vertices of each $3$-simplex of $K$ so that these orderings agree on common faces, and a decoration of $K$.

In otherwords: Any class $\alpha \in H_3 (\PGL_3 (\mathbb{C}))$ can be represented as $f_* [K]$ where $K$ is a quasi-simplicial complex such that $|K|-|K^{(0)}|$ is an oriented $3$-manifold, $[K] \in H_3 (|K|)$ is its fundamental class and $f:|K| \rightarrow \mathrm{B} \PGL_3 (\mathbb{C})$ is some map. 

This motivates the study of decorated $3$-cycles. From now on we fix $K$ a decorated oriented quasi-simplicial $3$-manifold together with an ordering of the vertices of each $3$-simplex of $K$. Let $N$
be the number of tetrahedra of $K$ and denote by $T_{\nu}$, $\nu =1 , \ldots , N$, these tetrahedra. We let 
$z_{ij}(T_{\nu})$ be the corresponding $z$-coordinates. We now describe the consistency relation on this coordinate in order to be able to glue back together the decorated tetrahedron

\subsection{Consistency relations} (cf. \cite{falbel}) Let $F$ be an internal face ($2$-dim cell) of $K$ and $T$, $T'$ be the tetrahedra attached to $F$. In order to fix notations, suppose that the vertex of $T$ are $1,2,3,4$ and that the face $F$ is $123$. Let $4'$ be the remaining vertex of $T'$. The face $F$ inherits two $3$-ratio from the decoration: first $z_{123}(T)$ as a face of $T$ and second $z_{132}(T')$ as a face of $T'$. But considering $F$ to be attached to $T$ or $T'$ only changes its orientation, not the flags at its vertex. So these two $3$-ratios are inverse. Hence we get the:

\medskip

\emph{(Face relation)} 
Let $T$ and $T'$ be two tetrahedra  of $K$ with a common face $(ijk)$ (oriented as a boundary of $T$), then $z_{ijk}(T)=\frac{1}{z_{ikj}(T')}$.

\medskip

\begin{figure}[ht]      
\begin{center}
\includegraphics[scale=.3]{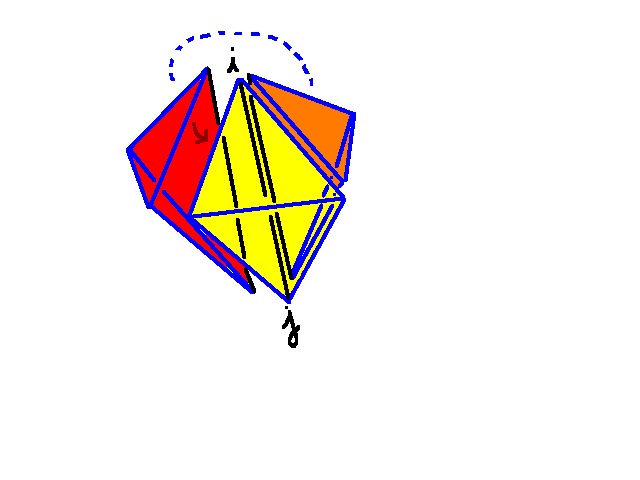}
\caption{tetrahedra sharing a common edge} \label{edge}
\end{center}
\end{figure}

We should add another compatibility condition to ensure that the edges are not singularities: we are going to compute the holonomy of a path in a decorated complex and we want it to be invariant by crossing the edges. One way to state the condition is the following one: let $T_1, \ldots , T_{\nu}$ be a sequence of tetrahedra sharing a common edge $ij$ and such that $ij$ is an inner edge of the subcomplex composed by the $T_{\mu}$'s (they are making looping around the edge, see figure \ref{edge}). Then we ask:

\medskip

\emph{(Edge condition)} $z_{ij}(T_1) \cdots z_{ij}(T_{\nu})=z_{ji}(T_1) \cdots z_{ji}(T_{\nu})=1$

\medskip

\subsection{Holonomy of a decoration}\label{holonomydecoration}

Recall from \S \ref{normal} that, once we have a configuration of $3$ generic flags $([x_i ] , [f_i ] )_{1 \leq i \leq 3} \in \C^r_2(\Fl)$
with triple ratio $X$, one defined a projective
coordinate system of $\P (\bC^3)$ as the one where the point $x_1$ has coordinates $[1:0 : 0]^t$, 
the point $x_2$ has coordinates $[0:0:1]^t$, the point $x_3$ has coordinates $[1:-1:1]^t$ and the intersection of $\mathrm {Ker} f_1$ and $\mathrm {Ker} f_2$ has coordinates $[0 : 1 : 0]^t$. The line $\mathrm {Ker} f_3$ then has coordinates $[X :  X+1 : 1]$. 

Given an oriented face we therefore get $3$ projective basis associated to the triples $(123)$, $(231)$ and $(312)$. 
The cyclic permutation of the flags induces the coordinate change given
by the matrix 
$$T(X) =  \left( 
\begin{matrix}
X & X+1 & 1 \\
-X & -X & 0 \\
X & 0 &  0
\end{matrix} \right).$$
Namely: if a point $p$ has coordinates $[u : v : w]^t$ in the basis associated to the triple $(123)$ it has coordinates $T(X) [u:v:w]^t$ in the basis 
associated to $(231)$.
 
\begin{lemma} \label{lem:hol}
If we have a tetrahedron of flags $(ijkl)$ with its $z$-coordinates, then the coordinate system
related to the triple $(ijk)$ is obtained from the coordinate system related to the triple $(ijl)$ 
by the coordinate change given by the matrix
$$E(z_{ij} , z_{ji}) = \left(
\begin{matrix}
z_{ji}^{-1} & 0 & 0\\
0 & 1 & 0 \\
0 & 0 & z_{ij} 
\end{matrix} \right).$$
\end{lemma}
\noindent
Beware that the orientation
of $(ijl)$ is not the one given by the tetrahedron.
\begin{proof} The matrix we are looking for fixes the flags $([x_1] , [f_1])$ and $([x_2] , [f_2])$ corresponding to the vertex $i$ and $j$. 
In particular it should be diagonal. Finally it should send $[x_4]$ to $[x_3]$.
But in the coordinate system associated to the triple $(ijk)$ the point $x_4$ in
the flag $([x_4] , [f_4])$ corresponding to the vertex $l$ has coordinates:
$$x_4 = [z_{ji} : -1 : z_{ij}^{-1} ]^t .$$
This proves the lemma.
\end{proof}

\subsection{} From this we can explicitly compute the holonomy of a path in the complex. For that let us put  three points in each face near the vertices denoting by $(ijk)$ the point in the face $ijk$ near $i$. As we have said before, each of these points corresponds to a projective basis of $\bC^3$. Each path can be deformed so that it decomposes in two types of steps (see figure \ref{fig:holon1}):
\begin{enumerate} 
\item a path  inside an oriented face $ijk$ from $(ijk)$ to $(jki)$,
\item a path through a tetrahedron $ijkl$ from $(ijk)$ to $(ijl)$ (i.e. turning left around the edge $ij$ oriented from $j$ to $i$).
\end{enumerate}
\begin{figure}[ht]      
\begin{center}
\def\svgwidth{12cm}
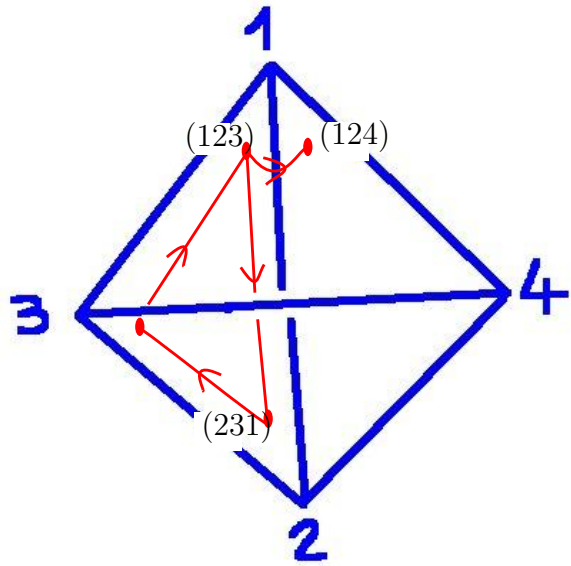
\caption{Two elementary steps for computing holonomy} \label{fig:holon1}
\end{center}
\end{figure}

Now the holonomy of the path is the coordinate change matrix so that: in case $1$, you have to left multiply by the matrix $T(z_{ijk})$ and in case $2$ by the matrix $E(z_{ij} , z_{ji})$. 

\subsection{} In particular the holonomy of the path turning left around an edge, i.e. the path $(ijk)\rightarrow(ijl)$, is given by 
\begin{equation} \label{Left}
L_{ij} = E(z_{ij} , z_{ji})= \left(
\begin{matrix}
z_{ji}^{-1} & 0 & 0\\
0 & 1 & 0 \\
0 & 0 & z_{ij} 
\end{matrix} \right).
\end{equation}
As an example which we will use latter on, one may also compute the holonomy of the path turning right around an edge, i.e. the path $(ilj)\rightarrow(ikj)$.
We consider the sequence of coordinate changes (see figure \ref{fig:holon2} for the path going from $(231)$ to $(241)$):
$$(ilj) \rightarrow (lji) \rightarrow (jil) \rightarrow (jik) \rightarrow (ikj).$$

\begin{figure}[ht]      
\begin{center}
\def\svgwidth{12cm}
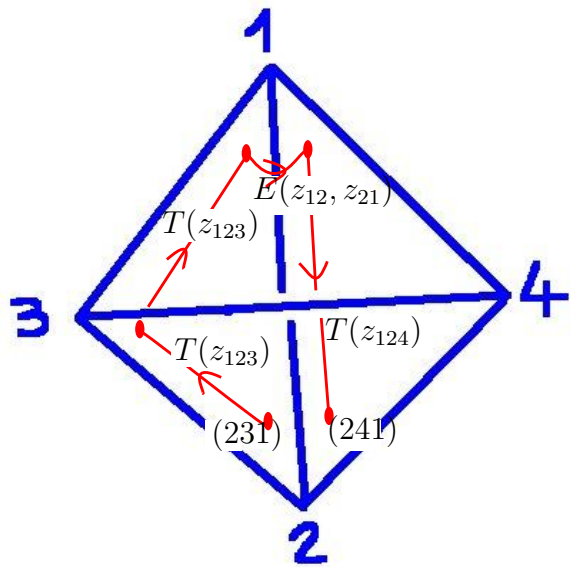
\caption{Turning right} \label{fig:holon2}
\end{center}
\end{figure}

The first two operations are cyclic permutations both given by the matrix $T(z_{ilj})$. It follows from lemma \ref{lem:hol} that the third is given by the matrix $E(z_{ji} , z_{ij})$. Finally the last operation
is again a cyclic permutation given by the matrix $T(z_{ikj})$. The coordinate change from the basis $(ilj)$ to $(ikj)$ is therefore given by 
$$T(z_{ikj}) E(z_{ji} , z_{ij}) T (z_{ilj})^2 = \begin{pmatrix}
           z_{ji} z_{ilj} & \star & \star  \\ & z_{ikj} & \star \\ & & \frac{z_{ikj}}{z_{ij}} \\
           \end{pmatrix}$$
Using $z_{ikj}=\frac 1{z_{ijk}}$, we get that the holonomy matrix, in $\PGL (3, \bC)$, of the path turning right around an edge $ij$ is 
\begin{equation} \label{Right}
R_{ij} = \begin{pmatrix}
           z_{ji}z_{ilj} z_{ijk}& \star & \star  \\ & 1 & \star \\ && \frac{1}{z_{ij}} \\
           \end{pmatrix} .
\end{equation}

\begin{remark}
\rm Beware that $L_{ij} R_{ij}$ is not the identity in $\PGL (3, \bC)$. This is due to the choices of orientations of the faces which prevents
$L_{ij} R_{ij}$ to be a matrix of coordinate change. When computing the holonomy of a path we therefore have to avoid backtracking.
\end{remark}

\subsection{Coordinates for the boundary of the complex} 
The boundary $\Sigma$ of the complex $K$ is a triangulated punctured surface. As in section \ref{decoration} and in \cite{FG} we associate to $\Sigma$ the set $I_{\Sigma}$ of the vertices of the (red) arrows of the triangulation of $\Sigma$ obtained using figure \ref{FG}. As in the preceeding section we set  $J_{\Sigma}^2 = \bZ^{I_{\Sigma}}$ 
and consider the skew-symmetric form $\Omega_{\Sigma}^2$ on $J_{\Sigma}^2$ introduced by Fock and Goncharov in \cite{FG}. Here again we let $J_{\Sigma}^* \subset (J_{\Sigma}^2)^*$ be the image of $J_{\Sigma}^2$ by the linear map $v \mapsto \Omega^2_{\Sigma} ( v, \cdot )$. 

\subsection{} \label{ref:512} The decoration of $K$ yields a decoration of the punctures of $\Sigma$ by flags, as in \cite{FG} and hence a point in $J_\Sigma^*$. Here is a more descriptive point of view, using the holonomy of the decoration of $K$: it provides $\Sigma$ with coordinates associated to each $\alpha \in I_{\Sigma}$. To each face we associate the face $z$-coordinate of the corresponding tetrahedra of $K$. To each oriented edge $ij$ of the triangulation of $\Sigma$ 
we associate the last eigenvalue of the holonomy of the path 
joining the two adjacent faces by turning left around $ij$ in $K$. It is equal to the product
$z_{ij} (T_1) \cdots z_{ij} (T_{\nu})$ where $T_1, \ldots , T_{\nu}$ is the sequence of tetrahedra sharing 
$ij$ as a common edge. 

We denote by $z_\Sigma$ the above defined element of $k^{\times} \otimes_{\bZ} J_{\Sigma}^* \left[ \frac12 \right]$. 

Note that when $K$ has a unipotent decoration, then the punctures are decorated by affine flags. We immediately get an element $a_{\Sigma} \in 
k^{\times} \otimes_{\bZ} J_{\Sigma}^2$ which projects onto $z_{\Sigma}$ in $k^{\times} \otimes_{\bZ} J_{\Sigma}^* \left[ \frac12 \right]$. Here again we have:
$$a_{\Sigma} \wedge_{\Omega^2_{\Sigma}} a_{\Sigma} = z_{\Sigma} \wedge_{\Omega^*_{\Sigma}} z_{\Sigma}.$$
The first expression is the $W$-element $W(\Sigma)$ associated to the decorated $\Sigma$ by Fock and Goncharov.

\subsection{Decoration and the pre-Bloch group}
Let $k$ be a field containing all the $z$-coordinates of the tetrahedra $T_{\nu}$, $\nu = 1 , \ldots , N$.
To any of these (ordered) tetrahedra we have associated an element $\beta (T_{\nu})\in \mathcal{P} (k)$. Set:
$$\beta (K) = \sum_{\nu} \beta (T_{\nu}) \in \mathcal{P} (k).$$

From now on we assume that for each vertex $v \in K^{(0)}$, $|L_v|$ is homeomorphic to either a torus or an annulus. We fix symplectic bases $(a_s , b_s )$ for each of the
tori components and we fix $c_r$ (resp. $d_r$) a generator of each homology group $H_1 (L_r)$ (resp. $H_1 (L_r , \partial L_r)$) where
the $L_r$'s are the annuli boundary components. We furthermore assume that the algebraic intersection number $\iota (c_r , d_r) = 1$. 

Each one of these homology elements may be represented as a path as in section \ref{holonomydecoration} which remains close to the associated vertex. So we may compute its holonomy using only matrices $L_{ij}$ and $R_{ij}$ : we will get an upper triangular matrix. More conceptually, the path is looping around a vertex decorated by a flag, so must preserve the flag. So it may be conjugated to an upper traingular matrix. Recall also that the diagonal part of a triangular matrix is invariant under conjugation by an upper-triangular matrix.

The following theorem computes $\delta (\beta (K))$ in terms of the {\it holonomy elements} $A_s$, $B_s$, $C_r$, $D_r$ and $A_s^*$, $B_s^*$, $C_r^*$, $D_r^*$ such that the holonomy matrices associated to $a_s$, $b_s$, $c_r$, $d_r$ have the following form in a basis adapted to the flag decorating the link (see also \S \ref{ss:linearizationholonomy} for a more explicit description):
$$\begin{pmatrix} \frac{1}{A_s^*} & * & *\\ 0 & 1 & *\\0&0&A_s\end{pmatrix}.$$

\begin{theorem} \label{thm:kabaya}
The invariant $\delta (\beta (K))$ only depends on the boundary coordinates $z_{\Sigma}$, $A_s$, $B_s$, $C_r$, $D_r$ and $A_s^*$, $B_s^*$, $C_r^*$, $D_r^*$. Moreover: 
\begin{enumerate}
\item If the decoration of $K$ is unipotent then $2 \delta (\beta (K)) = z_{\Sigma} \wedge_{\Omega_{\Sigma}^*} z_{\Sigma}$.
\item If $K$ is closed, i.e. $\Sigma = \emptyset$, and each link is a torus, we have the following formula for $3\delta (\beta (K))$:
$$ \sum_s \left( 2 A_s \wedge_{\bZ} B_s + 2 A_s^* \wedge_{\bZ} B_s^*+ A_s^* \wedge_{\bZ} B_s + A_s \wedge_{\bZ} B_s^* \right).
$$
\end{enumerate}\end{theorem}

Theorem \ref{thm:kabaya} generalizes several results known in the $\SL (2 , \bC)$-case, see Neumann \cite{Neuman} -- when $K$ is closed -- and Kabaya \cite{Kabaya} -- when all the connected components of $\Sigma$ are spheres with $3$ vertices. A related 
formula -- still in the $\SL(2,\bC)$-case -- is obtained by Bonahon \cite{Bonahon,Bonahon2}. One may extract from our proof a formula for the general case. 
Though it should be related to the Weil-Petersson form on $\partial M$ we are not able yet to explicit this relation.

\medskip
\noindent
{\it Remark.} Thanks to theorem \ref{thm:kabaya}, the fact that $\beta$ lies inside the Bloch group is a boundary condition (the only non-vanishing part is $\frac32 z_{\Sigma} \wedge_{\Omega_{\Sigma}^*} z_{\Sigma}$). As a consequence, if the boundary is empty, it will automatically belong to the Bloch group. Using the work of Suslin, it allows to construct geometrically any class in $K_3^{\rm ind}(k)$, empowering a remark of Fock and Goncharov, see \cite[Proposition 6.16]{FG2}.

\section{Some linear algebra and the unipotent case}

The goal of this section is to prove theorem \ref{thm:kabaya} when $K$ has
a unipotent decoration. Along the way, we lay down the first basis for the homological proof in the general case.

\subsection{} First let $(J^i , \Omega^i)$ ($i=\emptyset,2$) denote the orthogonal sum of the spaces $(J^i_{T_{\nu}} , \Omega^i)$. We denote by $e_{\alpha}^{\mu}$ the $e_{\alpha}$-element in $J_{T_{\mu}}^i$.

A decoration provides us with an element 
$$z \in \mathrm{Hom} (J , k^{\times}) \left[ \frac12 \right] \simeq k^{\times} \otimes_{\bZ} J^*\left[ \frac12 \right] = k^\times \otimes_\bZ\mathrm{Im} (p^*)\left[ \frac12 \right]$$ 
which satisfies the face and edge conditions.\footnote{Note that $z$ moreover satisfies the non-linear equations 
$$z_{ik} (T_{\nu}) = \frac{1}{1-z_{ij} (T_{\nu})}.$$} 
We first translate these two consistency relations into linear algebra.

Let $C^{\rm or}_1$ be the free $\bZ$-module 
generated by the oriented internal\footnote{Recall that our complex may have boundary.} $1$-simplices of $K$ and  $C_2$ the free $\bZ$-module 
generated by the internal $2$-faces of $K$.
Introduce the map
$$
F : C^{\rm or}_1+ C_2 \rightarrow  J^2
$$
defined by, for $\bar e_{ij}$ an internal oriented edge of $K$,  
$$
F(\bar e_{ij})= e_{ij}^{1} + \ldots + e_{ij}^{\nu} 
$$
where $T_{1}, \ldots , T_{\nu}$ is the sequence of tetrahedra sharing the edge $\bar e_{ij}$
such that $\bar e_{ij}$ is an inner edge of the subcomplex composed by the $T_{\mu}$'s and each $e_{ij}^{\mu}$ gets identified with the {\it oriented} edge $\bar e_{ij}$ in $K$ (recall figure \ref{edge}). And for 
a $2$-face $\bar e_{ijk}$, 
$$
F(\bar e_{ijk})=e_{ijk}^{\mu} + e_{ikj}^{\nu},
$$
where $\mu$ and $\nu$ index the two 
$3$-simplices having the common face  $\bar e_{ijk}$.
An element $z \in \mathrm{Hom} (J^2 ,k^{\times})$ satisfies the face and edge conditions if and only if it vanishes on $\mathrm{Im} (F)$. 

Let $(J_{\rm int}^2)^*$ be the subspace of $(J^2)^*$ generated by internal edges and faces of $K$. 

The dual map $F^* : (J^2)^* \rightarrow C_1^{\rm or} + C_2$ (here we identify $C_1^{\rm or} + C_2$ with its dual by using the canonical basis) is the ``projection map'':
$$(e_{\alpha}^{\mu})^* \mapsto \bar e_{\alpha}$$
when $(e_{\alpha}^{\mu})^* \in (J_{\rm int}^2)^*$ and maps  $(e_{\alpha}^{\mu})^*$ to $0$ if $(e_{\alpha}^{\mu})^* \notin (J_{\rm int}^2)^*$.

From the definitions we get the following:

\begin{lemma}
An element $z \in  k^{\times} \otimes_{\bZ} (J^2)^*$ satisfies the face and edge conditions if and only if
$$z \in k^{\times} \otimes_{\bZ} \mathrm{Ker} (F^* ) .$$
\end{lemma}

A decorated tetrahedra complex thus provides us with and element $z \in k^{\times} \otimes (J^* \cap \mathrm{Ker} (F^* )) \left[ \frac12 \right]$ and $\delta (\beta (K)) = \frac12 z \wedge_{\Omega^*} z$. 

\subsection{} In this section we assume that $K$ is equipped with a {\it unipotent} decoration. The boundary surface $\Sigma$ is then a union of ideally {\it triangulated} closed oriented\footnote{The orientation being induced by that of $K$.} surfaces with punctures decorated by affine flags in the sense of Fock and Goncharov \cite{FG}: the triangles are decorated by {\it affine} flags coordinates in such a way that the edge coordinates on the common edge of two triangles coincide. Each triangle being oriented we may define  the $W$-invariant:
$$W(\Sigma) = \sum_{\Delta} W_{\Delta}$$
where $W_{\Delta}$ is defined by \eqref{Wface}.\footnote{Note that in the case of $K=T$ the boundary of $T$ is a sphere with $4$ punctures and the definition of $W(T)$ in section \ref{decoration} matches this one.}

Recall from \S \ref{ref:512} that the unipotent decoration of $\Sigma$ provides us with an element $a_{\Sigma} \in k^{\times} \otimes_{\bZ} J^2_{\Sigma}$ which projects onto $z_{\Sigma} \in k^{\times} \otimes_{\bZ} J^*_{\Sigma}\left[ \frac12 \right]$. 
We have:\footnote{Note in particular that $W(\Sigma)$ only depends on the flag $z$-coordinates, see also \cite[Lem. 6.6]{FG2}. Moreover, in case $K=T$, we recover lemma \ref{deltaobeta}.}
$$W(\Sigma) = \frac12 a_{\Sigma} \wedge_{\Omega_{\Sigma}^2} a_{\Sigma} = \frac12 z_{\Sigma}\wedge_{\Omega_{\Sigma}^*} z_{\Sigma} .$$

We have already done the computations leading to the proof of the theorem \ref{thm:kabaya} in the unipotent case:

\begin{proposition}
In the unipotent case we have:
$$\delta (\beta (K)) = W( \Sigma ) .$$
\end{proposition}
\begin{proof} The proof is the same as that of \cite[Theorem 4.13]{FalbelWang}:
we compute $\sum \beta (T_{\nu} )$ for the tetrahedra complex using the $a$-coordinates as in \S \ref{acoord}. This gives a sum of $W$-invariants associated to the faces of the $T_{\nu}$'s. The terms corresponding to a common face between two tetrahedra appear with opposite sign. The sum of the remaining terms is precisely $W(\Sigma)$. 
\end{proof}

\subsection{}  
A unipotent decoration corresponds to a point $z \in k^{\times} \otimes (\mathrm{Im} ( p \circ F)) \left[ \frac12 \right]$. In 
\S \ref{ref:512} we therefore have defined a map 
$$k^{\times} \otimes (\mathrm{Im} ( p \circ F)) \left[ \frac12 \right] \to k^{\times}\otimes J^*_{\Sigma}\left[ \frac12 \right].$$
The following proposition states that this map respects the $2$-forms $\Omega^*$ and $\Omega^*_{\Sigma}$.

\begin{proposition} \label{Thm:unipotent}
In the unipotent case, $\Omega^*$ is the pullback of $\Omega^*_\Sigma$.
\end{proposition}
\begin{proof} 
We have seen that on each tetrahedron $p^*(\Omega^*(T))=\Omega^2(T)$. 

Since $\mathrm{Im} ( p \circ F)$ is the image  by $p$ of the subspace $\mathrm{Im}(F)$ of $J^2$, each face $f$ of $T$ is an oriented triangle with $a$-coordinates, so we define a $2$-form $\Omega^2(f,T)$ by the usual formula. If the face $f$ is internal between $T$ and $T'$, we have $\Omega^2(f,T)=-\Omega^2(f,T')$ as the only difference is the orientation of the face (and hence of its red triangulation, see figure \ref{FG}).

Moreover $p^*(\Omega^*)$ is the sum of the $\Omega^2(T)$. Hence it reduces to the sum on external faces of $\Omega^2(f,T)$, that is exactly $\Omega_\Sigma^2=p^*(\Omega^* )$.
\end{proof}

Our goal is now to extend this result beyond the unipotent case; to this end we develop a theory analogous to the one of Neumann-Zagier but in the $\PGL (3, \bC)$-case. We first treat in details the case where $K$ is closed.

\section{Neumann-Zagier bilinear relations for $\PGL (3, \bC)$}\label{sec:BilinRel}
 
A decorated tetrahedra complex provides us with and element $z \in k^{\times} \otimes (J^* \cap \mathrm{Ker} (F^* )) \left[ \frac12 \right]$ and $\delta (\beta (K)) = \frac12 z \wedge_{\Omega^*} z$. Our final goal is to compute this last expression. But here we first describe the right set up to state the generalization of proposition \ref{Thm:unipotent} 
to general -- non-unipotent -- decorations. This leads to a more precise version of theorem \ref{thm:kabaya}, see corollary \ref{corkab}.
We first deal with the case where $K$ is a (closed) $3$-cycle. We will later explain how to modify the definitions and proofs to deal with the general case.

\subsection{Coordinates on the boundary} Let $K$ be a quasi-simplicial triangulation of $M$. Assume that $K$ is closed so that $\Sigma = \emptyset$ and
each $|L_v|$ is a torus. We first define 
coordinates for $\partial M$ and a symplectic structure on these coordinates. 

Each torus boundary surface $S$ in the link of a vertex is triangulated by the traces of the tetrahedra; from this we build the CW-complex $\mathcal D$ whose edges consist of the inner edges of the first barycentric subdivision, see figure \ref{fig:celldecomposition}. We denote by $\mathcal D'$ the dual cell division Let $C_1 (\mathcal D) = C_{1} (\mathcal D ,\bZ)$ and $C_1 (\mathcal D') = C_{1} (\mathcal D' , \bZ)$ be the corresponding chain groups. Given two chains $c \in C_1 (\mathcal D)$ and $c' \in C_{1} (\mathcal D' )$ we denote by $\iota (c, c')$ the (integer) intersection number of $c$ and $c'$. This defines a bilinear form $\iota : C_1 (\mathcal D ) \times C_1 (\mathcal D' ) \to \bZ $ which induces the usual intersection form on $H_1 (S)$. In that
way $C_1 (\mathcal D' )$ is canonically isomorphic to the dual of $C_1 (\mathcal D)$. 

\begin{figure}[ht]      
\begin{center}
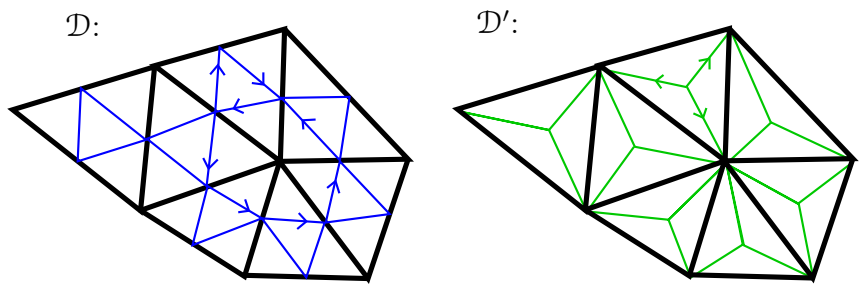
\caption{The two cell decompositions of the link} \label{fig:celldecomposition}
\end{center}
\end{figure}

\subsection{Goldman-Weil-Petersson form for tori} \label{ss:omegatori} Here we equip 
$$C_1 (\mathcal D , \bR^2) = C_1 (\mathcal D ) \otimes \bR^2$$
with the bilinear form $\omega$ defined by coupling the intersection form $\iota$ with the scalar product on $\bR^2$ seen as the space of roots of $\mathfrak{sl}(3)$ with its Killing form. We describe more precisely an integral version of this.

From now on we identify $\bR^2$ with the subspace $V = \{ (x_1 , x_2 , x_3)^t \in \bR^3 \; : \; x_1+x_2+x_3=0\}$ via 
$$\begin{pmatrix} 1\\0\end{pmatrix}\mapsto \begin{pmatrix}1\\-1\\0\end{pmatrix}\textrm{ and }\begin{pmatrix} 0\\1\end{pmatrix}\mapsto \begin{pmatrix}0\\1\\-1\end{pmatrix}.$$ 

We let $L \subset V$ be the standard lattice in $V$ where all three coordinates are in $\Z$. We identify it with $\Z^2$ using the above basis of $V$.   
The restriction of the usual euclidean product of $\bR^3$ gives a product, denoted $[,]$, on $V$ (the ``Killing form'')\footnote{In terms of roots of $\mathfrak{sl}(3)$, the choosen basis is, in usual notations, $e_1-e_2$, $e_2-e_3$.}. In other words, we have:
$$\left[\begin{pmatrix}1\\0\end{pmatrix},\begin{pmatrix}1\\0\end{pmatrix}\right]=\left[\begin{pmatrix}0\\1\end{pmatrix},\begin{pmatrix}0\\1\end{pmatrix}\right]=2\textrm{ and }\left[\begin{pmatrix}0\\1\end{pmatrix},\begin{pmatrix}1\\0\end{pmatrix}\right]=-1.$$
Identifying $V$ with $V^*$ using the scalar product $[,]$, the dual lattice $L^* \subset V^*$ becomes a lattice $L'$ in $V$; an element $y \in V$ belongs to $L'$
if and only if $[x,y] \in \Z$ for every $x \in L$. 

We let $C_1 (\mathcal D , L)$ and define $\omega = \iota \otimes [ \cdot , \cdot ] : C_1 (\mathcal D , L) \times C_1 (\mathcal D' , L') \rightarrow \bZ$ by the formula
$$\omega \left( c \otimes l , c' \otimes l' \right) = \iota (c,c') \left[ l , l' \right].$$
This induces a (symplectic) bilinear form on $H_1 (S, \bR^2)$ which we still denote by $\omega$. 
Note that $\omega$ identifies $C_1 (\mathcal D' , L')$ with the dual of $C_1 (\mathcal D , L)$. 

\begin{remark}
The canonical coupling $C_1 (\mathcal D , L) \times C^1 (\mathcal D , L^*) \rightarrow \Z$ identifies $C_1 (\mathcal D , L)^*$ with $C^1 (\mathcal D , L^*)$. This 
last space is naturally equipped with the ``Goldman-Weil-Petersson'' form $\textrm{wp}$, dual to $\omega$. Let $\langle , \rangle$ be the natural
scalar product on $V^*$ dual to $[,]$: letting $d:V\to V^*$ be the map defined by $d(v) = [v, \cdot ]$ we have
$\langle d (v) , d (v') \rangle = [v,v']$. In coordinates $d : \bR^2 \rightarrow \bR^2$ is given by
$$d\begin{pmatrix}x\\y\end{pmatrix}=\begin{pmatrix}2x-y\\2y-x\end{pmatrix}.$$ 
Idenfying $V^*$ with $\bR^2$ using the dual basis we have: 
$$\langle\begin{pmatrix}1\\0\end{pmatrix},\begin{pmatrix}1\\0\end{pmatrix}\rangle=\langle\begin{pmatrix}0\\1\end{pmatrix},\begin{pmatrix}0\\1\end{pmatrix}\rangle=\frac23\textrm{ and }\langle\begin{pmatrix}0\\1\end{pmatrix},\begin{pmatrix}1\\0\end{pmatrix}\rangle=\frac13.$$ 
On $H^1 (S , \bR^2)$ the bilinear form $\textrm{wp}$ induces a symplectic form -- the usual Goldman-Weil-Petersson symplectic form -- formally defined as the coupling of the cup-product and the scalar product $\langle , \rangle$.
\end{remark}

\subsection{} To any decoration $z\in k^{\times}\otimes (J^*\cap \mathrm{Ker}(F^*))\left[ \frac12 \right]$ we now explain how to associate an element 
$$R(z) \in \mathrm{Hom} (H_1(S , L), k^{\times} ) \left[ \frac12 \right].$$
We may represent any class in $H_1 (S, L)$ by an element $c \otimes \begin{pmatrix} n \\ m \end{pmatrix}$ in $C_1 (\mathcal D , L)$ where $c$ is a closed path in $S$ seen as the link of the corresponding vertex in the complex $K$. 
Using the decoration $z$ we may compute the holonomy of the loop $c$, as explained in \S \ref{holonomydecoration}. 
This vertex being equipped with a flag stabilized by this
holonomy, we may write it as an upper triangular matrix. Let $(\frac{1}{C^*},1,C)$ be the diagonal part. The application which maps 
$c \otimes \begin{pmatrix} n \\ m \end{pmatrix}$ to $C^{m} (C^*)^n$ is the announced element $R(z)$ of $k^{\times}\otimes H^1 (S , L^* )\left[ \frac12 \right]$.

\subsection{Linearization for a torus}\label{ss:linearizationholonomy}
In the preceeding paragraph we have constructed a map
$$R:k^{\times}\otimes J^*\cap \mathrm{Ker}(F^*)\left[ \frac12 \right] \to \mathrm{Hom} (H_1(S,L), k^{\times}) \left[ \frac12 \right].$$
As we have done before for consistency relations we now linearize this map. 

Let $h: C_1 (\mathcal D , L) \to J^2$ be the linear map defined on the elements $e \otimes \begin{pmatrix} n \\ m \end{pmatrix}$ of $C_1 (\mathcal D,L)$ by 
$$h \left( e \otimes \begin{pmatrix} n \\ m \end{pmatrix} \right) = 2m e_{ij}^{\mu} + 2n e_{ji}^{\mu} +  n (e_{ijk}^{\mu} + e_{ilj}^{\mu}) .$$
Here we see the edge $e$ as turning left around the edge $(ij)$ in the tetrahedron $T_{\mu} = (ijkl)$, see figure \ref{fig:h}.
 
\begin{figure}[ht]      
\begin{center}
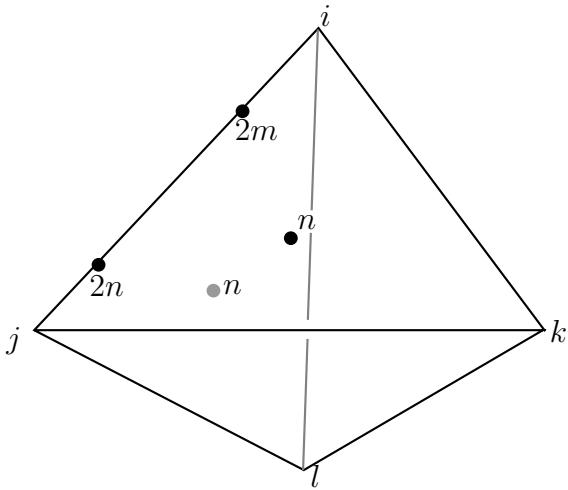
\caption{The map $h$} \label{fig:h}
\end{center}
\end{figure}
\begin{lemma}

Let $z \in k^{\times}\otimes (J^*\cap \mathrm{Ker}(F^*))\left[ \frac12 \right]$. Seeing $z$ as an element of $\mathrm{Hom} (J^2 , k^{\times}) \left[ \frac12 \right]$,
we have:
$$z \circ h = R(z)^2.$$
\end{lemma}
\begin{proof}
Let $c$ be an element in $H_1(S)$. 
 Recall that the torus is triangulated by the trace of the tetrahedra. To each triangle corresponds a tetrahedron $T_\mu$ and a vertex $i$ of this tetrahedron.  Now each vertex of the triangle corresponds to an edge $ij$ of the tetrahedron $T_\mu$ oriented from the vertex $j$ to $i$. Hence each edge of $\mathcal D$ may be canonically denoted by $c_{ij}^\mu$: it is the edge in the link of $i$ which turns left around the edge $ij$ of the tetrahedron $T_\mu$. We represent $c$ as a cycle $\bar c = \sum \pm c_{ij}^\mu$. The cycle $\bar c$ turns left around some edges, denoted by $e_{ij}^{\mu}$, and right around other edges, denoted by $e_{ij}^{\mu'}$. In other terms, we have $\bar c =\sum_\mu c_{ij}^\mu -\sum_{\mu'} c_{ij}^{\mu'}$.
Then, using the matrices $L_{ij}^\mu$ (\ref{Left}) and $R_{ij}^{\mu'}$ (\ref{Right}), we see that the diagonal part of the holonomy of $c$ is given by:
\begin{equation}\label{eq:C}
C=\frac{\prod z_{ij}^{\mu}}{\prod z_{ij}^{\mu'}}
\end{equation}
$$\textrm{ and }C^*=\frac{\prod z_{ji}^{\mu}}{\prod z_{ji}^{\mu'} z_{ijk}^{\mu'}z_{ilj}^{\mu'}}.$$
Let us simplify a bit the formula for $C^*$. Recall the face relation: if $T$ and $T'$ share the same face $ijk$, we have $z_{ijk}(T)z_{ikj}(T')=1$. Hence if our path $c$ was turning right before a face $F$ and continues after crossing $F$, the corresponding face coordinate simplifies in the product $\prod z_{ijk}^{\mu'}z_{ilj}^{\mu'}$. Let $\mathcal F$ be the set of faces (with multiplicity) at which $\alpha$ changes direction. For $F$ in $\mathcal F$, let $T$ be the tetrahedron containing $F$ in which $\alpha$ turns right. We consider $F$ oriented as a face of $T$ and denote $z_F$ its $3$-ratio. We then have
\begin{equation}\label{eq:C^*}
 C^*=\frac{\prod z_{ji}^{\mu}}{\prod z_{ji}^{\mu'}\prod_{\mathcal F} z_{F}}.
\end{equation}

Now $h\left(c\otimes \begin{pmatrix} 0 \\ 1 \end{pmatrix} \right)=2\sum e_{ij}^{\mu} - 2\sum e_{ij}^{\mu'}$, as turning right is the opposite to turning left. It proves (with equation \ref{eq:C}) that
$$z\circ h \left(c\otimes \begin{pmatrix} 0 \\ 1 \end{pmatrix} \right)=\left(\frac{\prod z_{ij}^{\mu}}{\prod z_{ij}^{\mu'}}\right)^2=C^2.$$

We have to do a bit more rewriting to check it for $c\otimes \begin{pmatrix} 1 \\ 0 \end{pmatrix}$. Indeed, we have:
$$h\left(c\otimes \begin{pmatrix} 1 \\ 0 \end{pmatrix} \right)=\sum_\mu (2 e_{ji}^{\mu} + e_{ijk}^\mu +e_{ilj}^{\mu} ) - \sum_{\mu'} (
2e_{ji}^{\mu'}+e_{ijk}^{\mu'} +e_{ilj}^{\mu'}),$$
so that:
$$
z\circ h\left(c\otimes \begin{pmatrix} 1 \\ 0 \end{pmatrix} \right)= \left(\frac{\prod z_{ji}^{\mu}}{\prod z_{ji}^{\mu'}}\right)^2\frac{\prod z_{ijk}^{\mu}z_{ilj}^{\mu}}{\prod z_{ijk}^{\mu'}z_{ilj}^{\mu'}} .
$$
For the same reason as before the ``internal faces'' simplify in the product $\prod z_{ijk}^{\mu}z_{ilj}^{\mu}$. Moreover, for $F\in \mathcal F$, it appears with the opposite orientation: indeed the orientation given to $F$ is the one given by the tetrahedron in which $\alpha$ turns \emph{left}. Hence the coordinate that shows up is $\frac 1{z_F}$. So the last formula rewrites:
$$
z\circ h\left(c\otimes \begin{pmatrix} 1 \\ 0 \end{pmatrix} \right)=\left(\frac{\prod z_{ji}^{\mu}}{\prod z_{ji}^{\mu'}}\right)^2\left(\frac{1}{\prod_{\mathcal F} z_{F}}\right)^2 = (C^*)^2,
$$
 which proves the lemma.
\end{proof}

Let $h^* : (J^2)^* \to C_1 (\mathcal D , L)^*$ be the map dual to $h$. Note that for any $e \in J^2$ and $c \in  C_1 (\mathcal D , L)$ we have
\begin{equation} \label{h*}
(h^* \circ p (e))( c) = p(e) (h( c)) = \Omega^2 (e , h( c)).
\end{equation}
Now composing $p$ with $h^*$ and identifying $C_1 (\mathcal D , L)^*$ with $C_1 ( \mathcal D' , L')$ using $\omega$ we get a map
$$g : J^2 \rightarrow C_1 (\mathcal D' , L')$$
and it follows from \eqref{h*} that for any $e \in J^2$ and $c \in J_T$ we have
\begin{equation} \label{g}
\omega (c, g (e)) = \Omega^2 (e , h(c )).
\end{equation}

In the following we let $J_{\partial M}= C_1 (\partial M , L)$ and $C_1 (\partial M ' , L')$
be the orthogonal sum of the $C_1 (\mathcal D , L)$'s and $C_1 (\mathcal D' , L')$'s for each torus link $S$. We abusively denote by $h : J_{\partial M} \to J^2$
and $g: J^2 \to C_1 (\partial M' , L')$ the maps defined above on each $T$.

\subsection{Homology of the complexes}\label{ss:defhomology}
First consider the composition of maps:
$$C_1^{\rm or} + C_2 \stackrel{F}{\rightarrow} J^2 \stackrel{p}{\rightarrow} (J^2)^* \stackrel{F^*}{\rightarrow} C_1^{\rm or} + C_2 .$$
By inspection one may check that $F^* \circ p \circ F = 0$. Here is a geometric way to figure this after tensorization by $k^{\times} \otimes \bZ\left[ \frac12 \right]$:
first note that if $z=p^*(a)$ then $a$ can be thought as a set of affine coordinates lifting of $z$. 
Now $a$ belongs to the image of $F$ exactly when these $a$-coordinates agree on elements of $J^2$ corresponding to common oriented edges (resp. common faces) of $K$. In such a case the decoration
of $K$ has a unipotent decoration lifting $z$. Finally the map $F^*$ computes the last eigenvalue of the holonomy matrix of paths going through and back a face (face relations) and of paths going around 
edges (edges relations). In case of a unipotent decoration these eigenvalues are trivial. This shows that 
$F^* \circ p \circ F=0$.

In particular, letting $G : J \rightarrow C_1^{\rm or} + C_2$ be the map induced by $F^* \circ p$ and 
$F' : C_1^{\rm or} 
+ C_2 \to J$ be the map
$F$ followed by the canonical projection from $J^2$ to $J$, we get a complex:
\begin{equation} \label{complexd}
C_1^{\rm or} + C_2 \stackrel{F'}{\rightarrow} J \stackrel{G}{\rightarrow} C_1^{\rm or} + C_2 .
\end{equation}
Similarly, letting $G^*= p \circ F$ and $(F')^*$ be the restriction of $F^*$ to $\mathrm{Im} (p) = J^*$ we get the dual complex:
\begin{equation} \label{complex}
C_1^{\rm or} + C_2 \stackrel{G^*}{\rightarrow} J^* \stackrel{(F')^*}{\rightarrow} C_1^{\rm or} + C_2 .
\end{equation}
We define the homology groups of these two complexes:
$$\mathcal{H} (J) = \mathrm{Ker} (G) / \mathrm{Im} (F') = \mathrm{Ker} (F^* \circ p) / (\mathrm{Im} (F) + \mathrm{Ker} ( p))$$ 
and 
$$\mathcal{H} (J^*) = \mathrm{Ker} ((F')^*) / \mathrm{Im} (G^*) = (\mathrm{Ker}(F^* ) \cap \mathrm{Im} ( p)) / \mathrm{Im} (p \circ F).$$ 
We note that:
$$\mathrm{Ker} (F') = \mathrm{Im} (G)^{\perp_{\Omega}} \mbox{ and } \mathrm{Ker} (G^*) = \mathrm{Im} ((F')^*)^{\perp_{\Omega^*}}.$$
The symplectic forms $\Omega$ and $\Omega^*$ thus induce skew-symmetric bilinear forms on 
$\mathcal{H} (J)$ and $\mathcal{H} (J^*)$. These spaces are obviously dual spaces and the bilinear
forms match through duality.

A decoration of $K$ provides us with an element $z \in k^{\times} \otimes \mathrm{Ker} ((F')^*)\left[ \frac12 \right]$. 
We already have dealt with the subspace $k^{\times} \otimes \mathrm{Im} (p\circ F)\left[ \frac12 \right]$ which corresponds to the unipotent decorations: in that case $\delta (\beta (K)) = 0$.
We thus conclude that $\delta (\beta (K))$ only depends on the image of $z$ in $k^{\times} \otimes_{\bZ} 
\mathcal{H} (J^*)\left[ \frac12 \right]$. We will describe this last space in terms of the homology of $\partial M$.

Let $Z_1 (\mathcal D , L)$ and $B_1 (\mathcal D, L)$ be the subspaces of cycles and boundaries in $C_1 (\mathcal D, L)$. The following lemma is easily checked by inspection.

\begin{lemma}
We have: 
$$h( Z_1 (\mathcal D, L)) \subset \mathrm{Ker} (F^* \circ p)$$
and 
$$h(B_1 (\mathcal D, L )) \subset \mathrm{Ker} (p) + \mathrm{Im} (F).$$
\end{lemma}
In particular $h$ induces a map $\bar{h} : H_1 (\mathcal D , L) \to \mathcal{H} (J)$ in homology. By duality, the map $g$ induces a map 
$\bar{g} : \mathcal{H} (J) \to H_1 (\partial M,\mathcal D' , L')$ as follows from:
\begin{lemma}
We have:
$$g(\mathrm{Ker}(F^* \circ p))\subset Z_1(\mathcal D',L'),$$
and
$$g(\mathrm{Ker} (p) + \mathrm{Im} (F))\subset B_1(\mathcal D',L').$$
\end{lemma}
\begin{proof}
First of all, $Z_1(\mathcal D',L')$ is the orthogonal of $B_1(\mathcal D,L)$ for the coupling $\omega$. Moreover, by definition of $g$, if $e\in\textrm{Ker} (F^* \circ p)$, we have: 
\begin{eqnarray*}
g(e)\in Z_1(\mathcal D',L') & \Leftrightarrow & \omega(B_1(\mathcal D,L),g(e))=0\\
& \Leftrightarrow & \Omega^2(h(B_1(\mathcal D,L)),e)=0.
\end{eqnarray*}
The last condition is given by the previous lemma. The second point is similar.
\end{proof}
Note that $H_1 (\mathcal D , L)$ and $H_1 (\mathcal D' , L')$ are canonically isomorphic so that we identified them (to $H_1 (\partial M , L)$) 
in the following. 
\begin{theorem} \label{thm:homologies}
\begin{enumerate}
\item The map $\bar{g} \circ \bar{h} : H_1 (\partial M , L ) \to H_1 (\partial M , L)$ is multiplication by~$4$.
\item Given $e \in \mathcal{H} (J)$ and $c \in H_1 (\partial M, L)$, we have 
$$\omega (c, \bar{g} (e)) = \Omega (e , \bar{h} ( c)).$$
\end{enumerate}
\end{theorem}

As a corollary, one understands the homology of the various complexes.

\begin{cor} \label{corhom}
The map $\bar h$ induces an isomorphism from $H_1 (\partial M, L)  \left[ \frac12 \right]$ to $\mathcal{H} (J) \left[ \frac12 \right]$. Moreover we have $\bar h^* \Omega = - 4 \omega$.
\end{cor}

\begin{cor}\label{corkab}
 The form $\Omega^*$ on $k^\times \otimes J^*\cap \mathrm{Ker}(F^*)\left[\frac12\right]$ is the pullback of $\mathrm{wp}$ on $H^1(\partial M,L^*)$ by the map $R$. 
\end{cor}

Theorem \ref{thm:kabaya} will follow from corollary \ref{corkab} and lemma \ref{lemfacile} (see section \ref{ss:kabaya2->kabaya} for an explicit computation). Corollary \ref{corkab}  is indeed the analog of proposition \ref{Thm:unipotent} in the closed case. We postpone the proof of theorem \ref{thm:homologies} until the next section and, in the remaining part of this section, deduce corollaries \ref{corhom} and \ref{corkab} from it.

\subsection{Proof of corollary \ref{corkab}} We first compute the dimension of the spaces $\mathcal H(J)$ and $\mathcal H(J^*)$. 
Recall that $l$ is the number of vertices in $K$.
\begin{lemma}\label{lem:dim}
The dimension of $\mathcal H(J)$ and $\mathcal H(J^*)$ is $4l$
\end{lemma}

\begin{proof} By the rank formula we have 
\begin{equation*} \label{dim1}
\dim J^2 = \dim \mathrm{Ker} (F^* \circ p) + \dim \mathrm{Im} (F^* \circ p)
\end{equation*}
and by definition we have
\begin{equation*} \label{dim2}
\dim \mathrm{Ker} (F^* \circ p) = \dim (\mathrm{Ker} ( p) +  \mathrm{Im} (F)) + \dim \mathcal H(J).
\end{equation*}
We obviously have: 
\begin{equation*} \label{dim3}
\dim (\mathrm{Ker} ( p) +  \mathrm{Im} (F)) = \dim \mathrm{Ker} ( p) + \dim \mathrm{Im} (F) - \dim (\mathrm {Ker}(p)\cap \mathrm{Im}(F))
\end{equation*}
and 
\begin{equation*} \label{dim4}
\dim \mathrm{Im} (F^* \circ p) = \dim \mathrm{Im} (F^*) - \dim (\mathrm{Im}(p)\cap \mathrm {Ker}(F^*)).
\end{equation*}
The map $F$ is injective and therefore $F^*$ is surjective. We conclude that 
$$\dim \mathrm{Im} (F) = \dim \mathrm{Im} (F^*) =\dim C_1^{\rm or} + \dim C_2.$$
But $\dim J^2 = 16N$, $\dim \mathrm{Ker} ( p) = 8N$, $\dim C_2 = 2N$ and, since the Euler caracteristic of $M$ is $0$, $\dim C_1^{\rm or} = 2N$. We are therefore reduced to prove that $\dim(\mathrm {Ker}(p)\cap \mathrm{Im}(F))=2l$. 
Restricted to a single tetrahedron $T_\mu$, the kernel of $p$ is generated by the elements $v_i^\mu=e_{ij}^\mu+e_{ik}^\mu+e_{il}^\mu$ and $w_i^\mu=e_{ji}^\mu+e_{ki}^\mu+e_{li}^\mu+e_{ijk}^\mu+e_{ilj}^\mu+e_{ikl}^\mu$ in $J^2(T_\mu)$ for $i$ a vertex of $T_\mu$ (see section \ref{ss:NZsympl}).
 
In $\mathrm{Im}(F)$, all the coordinates of $e_{ij}^\mu$ that projects on the same edge $\bar e_{ij}$ must be equal, as does the two coordinates of $e_{ijk}^{\mu}$ and $e_{ikj}^{\mu'}$ projecting on the same face. Hence, $\mathrm{Im}(F)\cap \mathrm {Ker}(p)$ is generated by the vectors $F(v_i)$ and $F(w_i)$ where $$v_i=\sum_{\bar e_{ij}\textrm{oriented edge from }i}\bar e_{ij} \textrm{ and }$$
 $$w_i= \sum_{\bar e_{ji}\textrm{ oriented edge toward }i}\bar e_{ji}+\sum_{\bar e_{ijk}\textrm{ a face containing s}i} \bar e_{ijk}.$$ One verifies easily that these vectors are free, proving the lemma. 
\end{proof}

Since it follows from theorem \ref{thm:homologies}~(1) that $\bar{h}$ has an inverse after tensorization by $\bZ \left[ \frac12 \right]$ we conclude from lemma \ref{lem:dim} that $\mathcal{H} (J) \left[ \frac12 \right]$ and $H_1 (\partial M, L) \left[ \frac12 \right]$ are isomorphic. Now \ref{thm:homologies}~(2) implies that $\bar{h}$
and $\bar{g}$ are adjoint maps w.r.t. the forms $\omega$ on $H_1 (\partial M, L) \left[ \frac12 \right]$ and $\Omega$ on $\mathcal{H} (J) \left[ \frac12 \right]$. The corollary follows.

The second corollary is merely a dual statement: recall from \ref{ss:linearizationholonomy} that the map $R^2$ is induced by the map
$h^* : J^* \to C^1 ( \mathcal D , L^*)$ dual to $h$. Now the map $c' \mapsto \omega ( \cdot , c')$ induces a symplectic isomorphism between
$(H_1 ( \partial M , L' ) , \omega )$ and $(H^1 (\partial M , L^* ) , \mathrm{wp})$. It therefore follows from corollary \ref{corhom} 
that the symplectic form $\Omega^*$ on $\mathcal{H} (J^*)$ is four times the pullback of $\mathrm{wp}$ by the 
map $\mathcal{H} (J^*) \to H^1 (\partial M , L^* )$ induced by $h^*$. Remembering that $h^*$ induces the {\it square} of $R$ the statement of corollary \ref{corkab} follows.

\section{Homologies and symplectic forms}

In this section we first prove theorem \ref{thm:homologies} (in the closed case). We then explain how to deduce theorem \ref{thm:kabaya} from it and its corollary \ref{corkab}.

\subsection{Proof of theorem \ref{thm:homologies}} \label{S81} We first compute $g \circ h : C_1 (\mathcal D , L) \to C_1 (\mathcal D ' , L')$ using \eqref{g}.
We work in a fixed tetraedron and therefore forget about the $\mu$'s. We denote by $c_{ij}$ the edge of $\mathcal D$ corresponding to a (left) turn around the edge
$e_{ij}$ and we denote by $c_{ij}'$ its dual edge in $\mathcal D'$, see figure \ref{fig:celldecomposition}. The following computations are straightforward:
$$\begin{array}{l}
\Omega \left(h\left(c_{ij} \otimes \begin{pmatrix} n \\ m \end{pmatrix} \right) , h \left( c_{ik} \otimes \begin{pmatrix} n' \\ m' \end{pmatrix} \right)  \right) 
= 2 \left[  \begin{pmatrix} n \\ m \end{pmatrix},  \begin{pmatrix} n' \\ m' \end{pmatrix} \right],\\
\Omega \left(h\left(c_{ij} \otimes \begin{pmatrix} n \\ m \end{pmatrix} \right) , h \left( c_{jk} \otimes \begin{pmatrix} n' \\ m' \end{pmatrix} \right)  \right) 
= -{2}{} \left[  \begin{pmatrix} (n+2m)/3 \\ (2n+m)/3 \end{pmatrix},  \begin{pmatrix} n' \\ m' \end{pmatrix} \right],\\
\Omega \left(h\left(c_{ij} \otimes \begin{pmatrix} n \\ m \end{pmatrix} \right) , h \left( c_{ji} \otimes \begin{pmatrix} n' \\ m' \end{pmatrix} \right)  \right) 
= 0, \\
\Omega \left(h\left(c_{ij} \otimes \begin{pmatrix} n \\ m \end{pmatrix} \right) , h \left( c_{ki} \otimes \begin{pmatrix} n' \\ m' \end{pmatrix} \right)  \right) 
= {2}{} \left[  \begin{pmatrix} (n+2m)/3 \\ (2n+m)/3 \end{pmatrix},  \begin{pmatrix} n' \\ m' \end{pmatrix} \right],
\end{array}$$
and so on...
Since it follows from \eqref{g} that 
\begin{multline*}
\omega \left( c \otimes \begin{pmatrix} n' \\ m' \end{pmatrix} , g \circ h \left(c_{ij} \otimes \begin{pmatrix} n \\ m \end{pmatrix} \right) \right) \\ = \Omega \left(
h \left(  c_{ij} \otimes \begin{pmatrix} n \\ m \end{pmatrix} \right) , h \left(c \otimes \begin{pmatrix} n' \\ m' \end{pmatrix} \right) \right)
\end{multline*}
we conclude that the element $g \circ h \left(c_{ij} \otimes \begin{pmatrix} n \\ m \end{pmatrix} \right)$ in $C_1(\mathcal D',L')$ is:
\begin{multline*}
g \circ h \left(c_{ij} \otimes \begin{pmatrix} n \\ m \end{pmatrix} \right) \\ = 2 (c_{ik}' - c_{il}') \otimes   \begin{pmatrix} n \\ m \end{pmatrix}
+ 2 (c_{ki}' - c_{kj}' + c_{jl}' - c_{jk} ' + c_{lj}' - c_{li} ' ) \otimes  \begin{pmatrix} (n+2m)/3 \\ (2n+m)/3 \end{pmatrix}.
\end{multline*}

Consider now a cycle $c=\sum c_{ij}^\mu$. We compute:
\begin{multline*}
 g\circ h\left( c\otimes \begin{pmatrix} n \\ m \end{pmatrix}\right)=\left(2\sum c_{ik}' - c_{il}'\right)\otimes   \begin{pmatrix} n \\ m \end{pmatrix}\\+\left(2 \sum c_{ki}' - c_{kj}' + c_{jl}' - c_{jk} ' + c_{lj}' - c_{li}'\right)\otimes \begin{pmatrix} (n+2m)/3 \\ (2n+m)/3 \end{pmatrix}
\end{multline*}
Interestingly, we are now reduced to a problem in the homology of $\partial M$ and the lattice $L$ does not play any role here. Indeed, the first assertion in theorem \ref{thm:homologies} follows from the following lemma. The second assertion of theorem \ref{thm:homologies} then follows from \eqref{g}.
\begin{lemma}
 \begin{itemize}
  \item The path $\sum c_{ik}' - c_{il}'$ is homologous to $2c$ in $H^1(\partial M)$,
  \item The path $ \sum c_{ki}' - c_{kj}' + c_{jl}' - c_{jk} ' + c_{lj}' - c_{li}'$ vanishes in $H^1(\partial M)$.
 \end{itemize}
\end{lemma}
This lemma is already proven by Neumann \cite[Lemma 4.3]{Neuman}. The proof is a careful inspection using figures \ref{fig:c->2c} and \ref{fig:far}. The first point is quite easy: the path $\sum c_{ik}' - c_{il}'$ is the boundary of a regular neighborhood of $c$. The second part is the ``far from the cusp'' contribution in Neumann's paper. We draw on figure \ref{fig:far} four tetrahedra sharing an edge (the edges are displayed in dotted lines). The blue path is the path $c$ in the upper link. The collection of green paths are the relative $ \sum c_{ki}' - c_{kj}' + c_{jl}' - c_{jk} ' + c_{lj}' - c_{li}'$ in the other links. It consists in a collection of boundaries.

\begin{figure}[ht]
\begin{center}
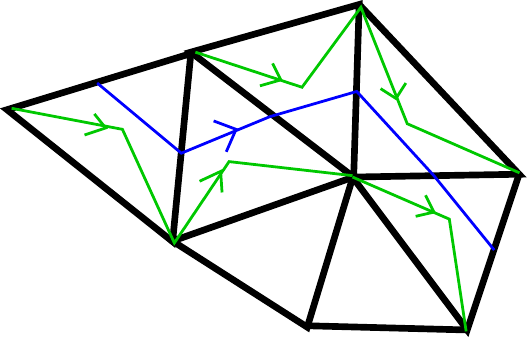
\caption{What happens inside the cusp: $c$ in blue and $g\circ h(c)$ in green.}\label{fig:c->2c}
\end{center}
\end{figure}

\begin{figure}[ht]
\begin{center}
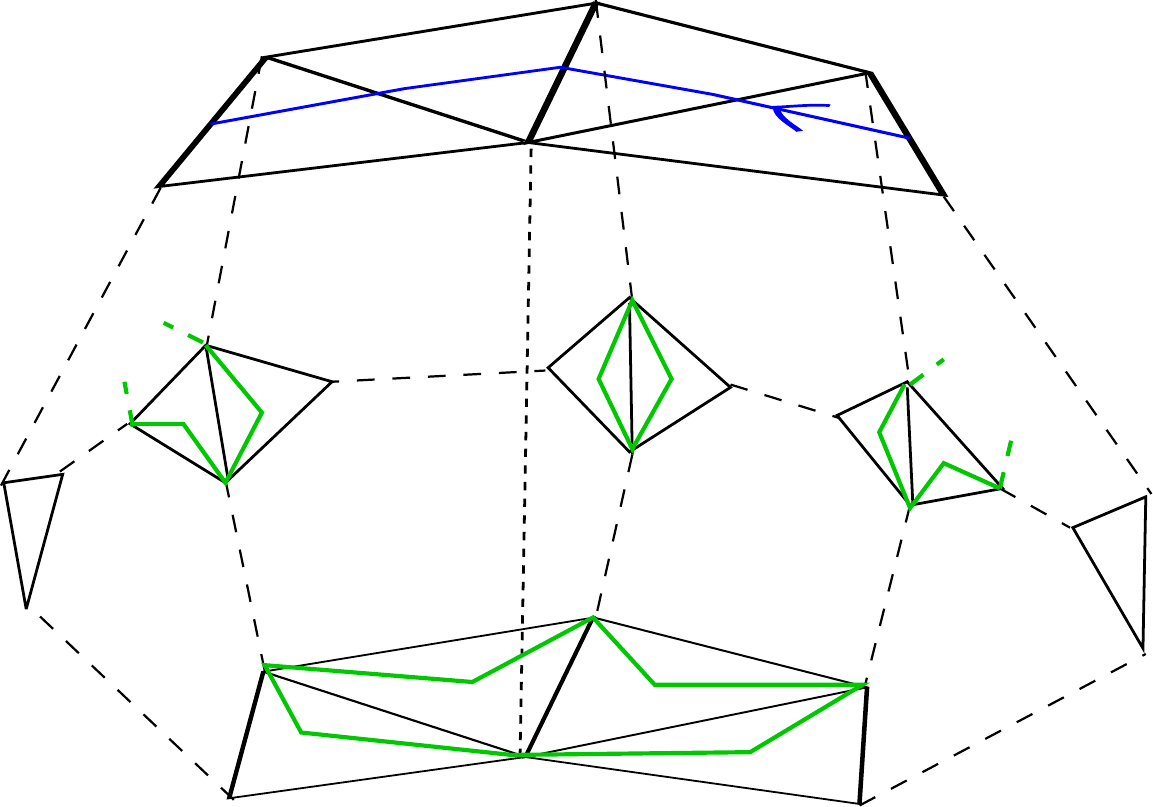
\caption{What happens far from the cusp}\label{fig:far}
\end{center}
\end{figure}

\subsection{Proof of theorem \ref{thm:kabaya} in the closed case}\label{ss:kabaya2->kabaya}

The theorem \ref{thm:kabaya} is now a corollary. Indeed, we have from lemma \ref{lemfacile} and corollary \ref{corkab}, if $z\in \mathcal H(J^*)$:
\begin{eqnarray*}
 3\delta(\beta(z))&=&\frac32 z\wedge_{\Omega^*}z\\
&=&\frac32 R(z)\wedge_{\mathrm{wp}} R(z) .
\end{eqnarray*}
It remains to compute the last quantity. Recall from the previous section the definition of $R(z)$: if a loop $c$ represents a class in homology, let $(\frac{1}{C^*},1,C)$ be the diagonal part of its holonomy. Then $R(z)$ applied to $c\otimes \begin{pmatrix} n\\m\end{pmatrix}$ equals $C^{m}(C^*)^n$. In other terms, denoting $[a_s]$ and $[b_s]$ the classes dual to $a_s$ and $b_s$, we have (see \S \ref{ss:omegatori}): 
$$R(z)=[a_s]\otimes \begin{pmatrix}A^*_s\\A_s\end{pmatrix}+[b_s]\otimes \begin{pmatrix}B_s^*\\B_s\end{pmatrix}.$$ Recall from \S \ref{ss:omegatori} that the form $\mathrm{wp}$ is the coupling of the cup product and the scalar product $\langle;\rangle$ on $\Z^2$.
Hence we conclude by:
\begin{eqnarray*}
3\delta(\beta(z))
&=&\sum_s 3\langle\begin{pmatrix}A_s^*,A_s\end{pmatrix},\begin{pmatrix}B_s^*,B_s\end{pmatrix}\rangle\\
&=&\sum_s 2A_s\wedge B_s+2A_s^*\wedge B_s^*+A_s^*\wedge B_s+A_s\wedge B_s^*.
\end{eqnarray*}

\section{Extension to the general case} We consider now the case of a  complex $K$ with boundary and explain how the preceeding proof of theorem \ref{thm:kabaya} shall be adapted to deal with it. Recall that the boundary of $K-K^{(0)}$ decomposes as the union of a triangulated surface $\Sigma$ and the links. The latter are further decomposed as tori links $S_s$ and annuli links $L_r$. We proceed as in the closed case and indicate the
modifications to be done. For simplicity we suppose that $k = \bC$. 

\subsection{} We denote by $C_1^{\rm or}+C_2$ the $\Z$-module generated by {\it internal} (oriented) edges and faces.
A parabolic decoration of $K$ gives a parabolic decoration of $\Sigma$, i.e. an element $z_\Sigma\in k^{\times}\otimes_\Z (J^2_{\Sigma})^*\left[ \frac12 \right]$, whose interpretation is that one may glue the decorated surface $\Sigma$ to the decorated complex fulfilling the consistency relations. More precisely, if $e_{\alpha}$ is a basis vector of $J^2_\Sigma$, one defines the $e_\alpha^*$ component of $z_\Sigma$ by:
$$z_\alpha^\Sigma \prod_\nu z_\alpha^\nu=1,$$
where the product is over all the $e_\alpha^\nu$ identified with $e_\alpha$. As usual we will rather consider the corresponding linear map:
$$h_{\Sigma} : J_{\Sigma}^2 \rightarrow J^2 ; e_{\alpha} \mapsto -\sum_{\nu} e_{\alpha}^{\nu}$$
as well as the dual map $h_{\Sigma}^* : (J^2)^* \to (J_{\Sigma}^2 )^*$. Note that if $e_{\alpha}^{\nu}$ in $J^2$ corresponds to an {\it internal} edge or face 
then $h^* ((e_{\alpha}^{\nu})^*) = 0$ whereas if it correspond to a boundary element $e_{\alpha} \in J_{\Sigma}^2$ we have:  $h^* ((e_{\alpha}^{\nu})^*) = - e_{\alpha}^*$. In particular one easily check that the following diagram is commutative:
\begin{equation} \label{cd}
\begin{CD}
J^2 @>p>> (J^2)^* \\
@AhAA   @VVh^*V \\
J_{\Sigma}^2 @>p_{\Sigma}>> (J_{\Sigma}^2)^* 
\end{CD}
\end{equation}
Recall that $J_{\Sigma} = J_{\Sigma}^2 / \mathrm{Ker} (p_{\Sigma})$.

The cell decomposition $\mathcal D$ is now defined for every cusp, each of which being either a torus or an annulus. In the latter case we may consider cycles relative
to the boundary. We denote by $Z_1^{\mathrm{rel}}(\mathcal D,L)$, resp. $Z_1^{\mathrm{rel}}(\mathcal D',L')$, the subspace of relative cycles in 
$C_1 (\mathcal{D} , L)$, resp. $C_1 (\mathcal{D}' , L')$. It is the orthogonal of $B_1(\mathcal D',L')$, resp. $B_1(\mathcal D,L)$, w.r.t. to the form $\omega$ defined as above, see \S \ref{ss:omegatori}. 

\subsection{} We now set 
$$J_{\partial M}^2 = J_{\Sigma}^2 \oplus C_1 (\mathcal D , L), \quad (J_{\partial M}^2)' = J_{\Sigma}^2 \oplus C_1 (\mathcal D' , L')$$
and let 
$$\Omega^2_{\partial M} : J_{\partial M}^2 \times (J_{\partial M}^2)' \to \bZ$$
be the bilinear coupling obtained as the orthogonal sum of $\Omega_{\Sigma}^2$ and $\omega$. As above it corresponds to these data
the map $p_{\partial M} : J_{\partial M}^2 \rightarrow ((J_{\partial M}^2)')^*$, $p_{\partial M} ( c)= \Omega^2_{\partial M} ( c, \cdot )$, as well as the spaces
$$J_{\partial M}  = J_{\partial M}^2 / \mathrm{Ker} (p_{\partial M}) = J_{\Sigma} \oplus C_1 (\mathcal D , L)$$
and 
$$(J_{\partial M}')^* = \mathrm{Im} (p_{\partial M}) = J_{\Sigma}^* \oplus C_1 (\mathcal D' , L').$$ 
The bilinear coupling induces a canonical perfect coupling
$$\Omega_{\partial M} : J_{\partial M}  \times J_{\partial M}' \to \bZ$$
which identifies $J_{\partial M}^*$ with $J_{\partial M}'$.

\subsection{} As in the closed case (see \S \ref{ss:linearizationholonomy}) the linearization of the holonomy yields an extension of $h_{\Sigma}$ to a map 
$h : J_{\partial M}^2 \to J^2$. We then have the following diagram:
$$
\begin{CD}
C_1^{\rm or}+C_2 @>F>> J^2 @>p>> (J^2)^* @>F^*>> C_1^{\rm or}+C_2 \\
@. @AhAA @Vh^*VV @. \\
@. J_{\partial M}^2 @. (J_{\partial M}^2)^* @.
\end{CD}
$$
Now it follows from \eqref{cd} that the image of $h^* \circ p$ is contained in $J_{\partial M}^*$. Identifying it with $J_{\partial M}'$ using $\Omega_{\partial M}$
we get a map $g : J^2 \to  J_{\partial M}'$. As in the closed case, for any $c \in J^2_{\partial M}$ and $e \in J^2$, we have:\footnote{Here we abusively use the same notation for $c$ and its image in $J_{\partial M}$.}
\begin{equation} \label{gomega}
\Omega_{\partial M} (c , g(e)) = \Omega^2 (e , h( c)).
\end{equation}
We moreover have the following inclusions:
\begin{itemize}
 \item $h(J_{\Sigma}^2 \oplus Z_1^{\rm rel}(\mathcal D,L))\subset \mathrm{Ker}(F^*\circ p)$,
 \item $h(J_{\Sigma}^2 \oplus B_1(\mathcal D,L))\subset \mathrm{Im}(F)+{\rm Ker}(p )$.
\end{itemize}
Denoting
$$\mathcal{H}_{\partial M} = (J_{\Sigma}^2 \oplus Z_1^{\rm rel}(\mathcal D,L))/h^{-1} (\mathrm{Im}(F)+{\rm Ker} ( p))$$
and 
$$\mathcal{H}_{\partial M} ' = (J_{\Sigma} \oplus Z_1^{\rm rel}(\mathcal D',L'))/g (\mathrm{Im}(F)+{\rm Ker} ( p)),$$
we conclude that the maps $h$ and $g$ induce maps
$$\bar h : \mathcal{H}_{\partial M} \to \mathcal{H} (J) \quad \mbox{ and } \quad \bar g : \mathcal{H} (J) \to \mathcal{H}_{\partial M} ' .$$
It furthermore follows from \eqref{gomega} that $\Omega_{\partial M}$ induces a bilinear coupling
$$\bar \Omega_{\partial M} : \mathcal{H}_{\partial M} \times \mathcal{H}_{\partial M} ' \rightarrow \bZ.$$

\begin{lemma} \label{l94}
The bilinear coupling $\bar \Omega_{\partial M}$ is non-degenerate.
\end{lemma}
\begin{proof}
Denote by $\partial M\setminus \Sigma$ the union of the links (tori and annuli). The quotient $J_{\Sigma}$ of $J_{\Sigma}^2$ naturally identifies 
with the quotient of $\mathrm{Im} (F) + h (J_{\Sigma}^2)$ by $\mathrm{Im} (F) + \mathrm{Ker} (p )$. Note that the former identifies with the  image of the
$\bZ$-module generated by {\it all} (oriented) edges and faces of $K$ into $J^2$. 
We then have two short exact sequences
$$0 \to J_{\Sigma} \to \mathcal{H}_{\partial M} \to H_1 ^{\rm rel} (\partial M \setminus \Sigma , L) \to 0$$
and
$$0 \to H_1  (\partial M \setminus \Sigma , L') \to \mathcal{H}_{\partial M}' \to J_{\Sigma} \to 0.$$
These are in duality w.r.t. $\Omega_{\partial M}$. Moreover this duality yields $\Omega_{\Sigma}$ on the product $J_{\Sigma} \times J_{\Sigma}$ and 
the intersection form, coupled with $[,]$, on $H_1 ^{\rm rel} (\partial M \setminus \Sigma , L) \times H_1  (\partial M \setminus \Sigma , L)$. Since both are non-degenerate this proves the lemma.
\end{proof}

It now follows from \eqref{gomega} that $\Omega_{\partial M} ( \cdot , g \circ h (\cdot )) = \Omega^2 (h (\cdot ) , h (\cdot ))$. And computations similar to \S \ref{S81} show that the right-hand side has a trivial kernel on $\mathcal{H}_{\partial M}$. 
The coupling $\bar \Omega_{\partial M}$ being non-degenerate we conclude that $\bar h$ is injective. As in the closed case, we may furthermore compute the dimension of $\mathcal H(J)$. Let $\nu_t$ be the number of tori and $\nu_a$ be the number of annuli. Then, computing the Euler characteristic of the double of $K$ along $\Sigma$, the proof of lemma \ref{lem:dim} yields the following:
\begin{lemma}
 The dimension of $\mathcal H(J)$ is $4\nu_t+2\nu_a+\dim(J_\Sigma)$.
\end{lemma}
This is easily seen to be the same as both the dimensions of $\mathcal{H}_{\partial M}$ and $\mathcal{H}_{\partial M}'$, see the proof of lemma \ref{l94}. Over $\bC$ the maps $\bar h $ and $\bar g$ are therefore invertible and we conclude that the form $\Omega$ on $J$ induces a form $\bar \Omega$ on $\mathcal{H} (J)$ such that 
$$\bar \Omega_{\partial M} ( c, \bar g (e))  = \bar \Omega ( e , \bar h ( c)).$$
In particular $\bar \Omega$ is determined by $\Omega_{\partial M}$ and the invariant $\delta  (\beta (K))$ only depends on the boundary coordinates. This concludes the proof of theorem \ref{thm:kabaya}.

\section{Examples}\label{examples}

In this section we describe the complement of the figure eight knot obtained by gluing two tetrahedra.
Let $z_{ij}$ and $w_{ij}$ be the coordinates associated to the edge $ij$ of each of them.

\begin{figure}[ht]
\begin{center}
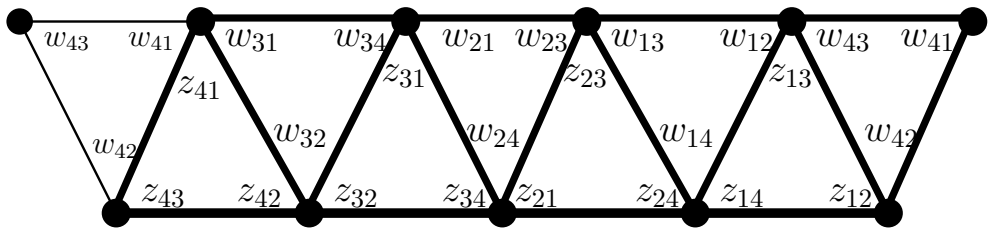
\caption{The link at the boundary for the figure eight knot}\label{figure:holo}
\end{center}
\end{figure}

The edge equations are:
\begin{itemize}
\item $z_{12}w_{12}z_{13}w_{43}z_{43}w_{42}=1$
\item  $z_{21}w_{21}z_{31}w_{34}z_{34}w_{24}=1$
\item $z_{42}w_{32}z_{32}w_{31}z_{41}w_{41}=1$
\item $z_{24}w_{23}z_{23}w_{13}z_{14}w_{14}=1$
\end{itemize}

The face equations are:

\begin{itemize}
\item $z_{13}z_{43}z_{23}w_{14}w_{34}w_{24}=1$
\item  $z_{14}z_{24}z_{34}w_{21}w_{41}w_{31}=1$
\item $z_{12}z_{42}z_{32}w_{13}w_{43}w_{23}=1$
\item $z_{21}z_{31}z_{41}w_{12}w_{32}w_{42}=1$
\end{itemize}

And the holonomies are:

\begin{itemize}
\item $A= z_{41}\frac{1}{w_{32}} z_{31}\frac{1}{w_{24}}z_{23}\frac{1}{w_{14}}z_{13}\frac{1}{w_{41}}$
\item $B=  z_{43}\frac{1}{w_{41}}$
\item $A^*= \frac{1}{z_{14}}\frac{w_{14}w_{41}}{w_{32}} \frac{1}{z_{13}}
\frac{w_{13}w_{31}}{w_{24}}\frac{1}{z_{32}}\frac{w_{23}w_{32}}{w_{14}}\frac{1}{z_{31}}\frac{w_{31}w_{13}}{w_{42}}$
\item $B^*=  \frac{1}{z_{34}}\frac{w_{23}w_{32}}{w_{41}}$
\end{itemize}

If $A=B=A^*=B^*=1$ the solutions of the equations correspond to unipotent 
structures.  The complete hyperbolic 
structure on the complement of the figure eight knot determines a solution 
of the above equations.
In fact, in that case, if $\omega=\frac{1+i\sqrt{3}}{2}$ then
$$
z_{12}=z_{21}=z_{34}=z_{43}=w_{12}=w_{21}=w_{34}=w_{43}=\omega
$$
is a solution the equations as obtained in \cite{Thurston}. 

The spherical CR
structures with unipotent boundary holonomy were obtained in 
\cite{falbeleight} as the following solutions 
(up to conjugation):
$$
z_{12}=\bar z_{21}=z_{34}=\bar z_{43}=w_{12}=\bar w_{21}=w_{34}=\bar w_{43}=\omega,
$$
$$
z_{12} = \frac{5-i\sqrt{7}}{4},z_{21} = \frac{3-i\sqrt{7}}{8}, z_{34} =\frac{5+i\sqrt{7}}{4},z_{43}  = \frac{3+i\sqrt{7}}{8}
$$
$$
w_{12} = \frac{3-i\sqrt{7}}{8},w_{21} = \frac{5-i\sqrt{7}}{4},
w_{34}  = \frac{3+i\sqrt{7}}{8},w_{43} =\frac{5+i\sqrt{7}}{4}
$$
and
$$
z_{12} = \frac{-1+i\sqrt{7}}{4},z_{21} = \frac{3-i\sqrt{7}}{2}, 
z_{34} =\frac{-1-i\sqrt{7}}{4},z_{43}  = \frac{3+i\sqrt{7}}{2}
$$
$$
w_{12} = \frac{3+i\sqrt{7}}{2},w_{21} = \frac{-1-i\sqrt{7}}{4} 
w_{34}  = \frac{3-i\sqrt{7}}{2},w_{43} =\frac{-1+i\sqrt{7}}{4}
$$

The first solution above corresponds to a discrete representation 
of the fundamental group of the complement of the figure eight knot in 
$\mathrm{PU}(2,1)$ with faithful boundary holonomy.  Moreover, its action on complex hyperbolic space has limit set the full boundary sphere.  The other solutions 
have cyclic boundary holonomy.

 We will call
these solutions standard structures on the complement of the figure eight knot. Recently, P.-V. Koseleff
proved that they are the only solutions to the equations:

\begin{proposition} The only unipotent flag $\mathrm{SL}(3,\bC)$-structures
on the complement of the figure eight knot are the standard structures.
\end{proposition}

\section{Applications}\label{applications}

\subsection{Volumes of decorated tetrahedra complex}

A decorated closed tetrahedra complex $K$ provides us with an element $z \in \bC^{\times} \otimes_{\Z} J^*\left[ \frac12 \right]$ which satisfies
the face and edge conditions as well as the non-linear equations
$$z_{ik} (T_{\nu}) = \frac{1}{1-z_{ij} (T_{\nu})}.$$
Let $X = \bC^{\times} \otimes_{\Z} J^*\left[ \frac12 \right]$; this is a
complex variety. 

Following \S \ref{par:vol} we define the {\it volume} of $K$ as:
\begin{equation} \label{eq:defvol}
\mathrm{Vol} (K) = \frac{1}{4} D (\beta (K)).
\end{equation}
This defines a real analytic function on $X$:
$$\mathrm{Vol} : X \rightarrow \bC.$$

Let $\mathcal{F} (X)^{\times}$ be the group of invertible real analytic functions on $X$ and 
$\Omega^1 (X)$ the space of real analytic $1$-form on $X$. The holonomy
elements $A_s$ , $A_s^*$ and $B_s$, $B_s^*$ of theorem 
\ref{thm:kabaya} define elements 
in $\mathcal{F}(X)^{\times}$. Now there is a map $\mathrm{Im} ( d \log \wedge_{\Z} \log) : \mathcal{F}(X)^{\times} \wedge_{\Z} \mathcal{F}(X)^{\times} \rightarrow \Omega^1 (X)$ defined by:
$$\mathrm{Im} ( d \log \wedge_{\Z} \log) (f \wedge_{\Z} g )= \mathrm{Im} \left( \log |g| \cdot d(\log f) - \log |f| \cdot d(\log g) \right).$$

Following Neumann-Zagier \cite{NeumanZagier} we want to compute the variation of $\mathrm{Vol} (K)$ as we vary $z \in X$. 
Equivalently we compute $d\mathrm{Vol} \in \Omega^1 (X)$ in the following:
\begin{proposition} \label{prop:last}
We have:
\begin{multline*}
d\mathrm{Vol} \\ = \frac{1}{12} \sum_s \mathrm{Im} ( d \log \wedge_{\Z} \log) (2A_s \wedge_{\Z} B_s + 2A_s^* \wedge_{\Z} B_s^*+A_s^*\wedge_\Z B_s +A_s\wedge_\Z B_s^*).
\end{multline*}
\end{proposition} 
\begin{proof} The derivatives of $D(z)$ are elementary functions:
\begin{equation} \label{eq:derD}
\frac{\partial D}{\partial z} = \frac{i}{2} \left( \frac{\log |1-z|}{z} + \frac{\log |z|}{1-z} \right), \quad \frac{\partial D}{\partial \overline{z}} = -\frac{i}{2} \left( \frac{\log |1-z|}{\overline{z}} + \frac{\log |z|}{1-\overline{z}} \right).
\end{equation}
Assume that the parameter $z \in \bC^*$ is varying in dependence on a single variable $t$. Then:
\begin{equation*}
\begin{split}
\frac{d}{dt} D(z_t)  & = \frac{i}{2} \left[ \left( \frac{\log |1-z|}{z} + \frac{\log |z|}{1-z} \right) \frac{dz}{dt}  - \left( \frac{\log |1-z|}{\overline{z}} + \frac{\log |z|}{1-\overline{z}} \right) \frac{d\overline{z}}{dt} \right] \\
& =  \mathrm{Im} \left( \left( \frac{d}{dt} \log (z) \right)  \log |1-z| - \left( \frac{d}{dt} \log (1-z)\right) \log |z|  \right).  
\end{split}
\end{equation*}
In otherwords: $dD =  \mathrm{Im} ( d \log \wedge_{\Z} \log) (z \wedge_{\bZ} (1-z))$. And proposition
\ref{prop:last} follows from theorem \ref{thm:kabaya} and \eqref{eq:defvol}.
\end{proof}

\medskip
\noindent
{\it Remark.} Proposition \ref{prop:last} implies in particular that the variation of the volume only depends
on the contribution of the boundary. Specializing to the hyperbolic case we recover the result of Neumann-Zagier \cite{NeumanZagier}, see also Bonahon \cite[Theorem 3]{Bonahon2}.

\subsection{Weil-Petersson forms} Let $k$ be an arbitrary field. The Milnor group $K_2 (k)$
is the cokernel of $\delta : \mathcal{P} (k) \rightarrow k^{\times} \wedge_{\Z} k^{\times}$. 

Let $X_{\Sigma} = \bC^{\times} \otimes_{\Z} J^*_{\Sigma}\left[ \frac12 \right]$; it is a complex manifold. As above
we may consider the field $\mathcal{F}(X_{\Sigma})^{\times}$; we let $\Omega^2_{\rm hol} (X_{\Sigma})$
denote the space of holomorphic $2$-forms on $X_{\Sigma}$. The element $z_{\Sigma}$
defines an element in $\mathcal{F}(X_{\Sigma})^{\times}$. We still denote the projection of
$z_{\Sigma} \wedge_{\Omega_{\Sigma}^*} z_{\Sigma}$ into $K_2 ( \mathcal{F}(X_{\Sigma})^{\times})$. 

Now, since $d\log \wedge_{\Z} d\log ((1-f) 
\wedge_{\Z} f) =0$, there is a group homomorphism:
$$d \log \wedge_{\Z} d \log : K_2 ( \mathcal{F}(X)^{\times}_{\Sigma}) \rightarrow \Omega^2 (X_{\Sigma}), \quad f \wedge_{\Z} g \mapsto 
d \log (f) \wedge_{\Z} d \log (g).$$

In the hyperbolic case and when the decoration is unipotent, Fock and Goncharov \cite{FG} prove
that
$$\frac12 d \log z_{\Sigma} \wedge_{\Omega_{\Sigma}^*} d \log z_{\Sigma} = d \log \wedge_{\Z} d \log (W(\Sigma ))$$
is the Weil-Petersson form. Although expected, the analogous statement in the $\SL(3)$-case 
seems to be open. In any case theorem \ref{thm:kabaya} implies that this form vanishes, equivalently
the ``Weil-Petersson forms'' corresponding to the different components of $\Sigma$ add up to zero.

\bibliography{bibli}

\bibliographystyle{plain}

\end{document}